\documentclass{siamltex}
\usepackage{amsmath,latexsym,amsfonts,amssymb,mathrsfs,verbatim,euscript}

\usepackage{type1cm}        
\usepackage{graphicx}       
\usepackage{algpseudocode,algorithm}
\floatname{algorithm}{Pseudocode}

\newtheorem{assumption}{Assumption}

\newcommand{\beqn}{\begin{equation}}
\newcommand{\eeqn}{\end{equation}}

\newcommand{\fa}{{\mathfrak a}}

\def\R{\mathbb{R}}
\def\N{\mathbb{N}}

\newcommand{\cC}{\mathcal{C}}

\newcommand{\cK}{\mathcal{K}}

\newcommand{\cO}{\mathcal{O}}
\newcommand{\cS}{{U}}
\newcommand{\cX}{\mathcal{X}}
\newcommand{\bcX}{\mathcal{X}}
\newcommand{\cY}{\mathcal{Y}}
\newcommand{\bcY}{\mathcal{Y}}

\newcommand{\sbinom}[2]{{\textstyle{\binom{#1}{#2}}}}
\newcommand{\satop}[2]{\stackrel{\scriptstyle{#1}}{\scriptstyle{#2}}}

\newcommand{\bsUpsilon}{{\boldsymbol{\Upsilon}}}

\newcommand{\bsalpha}{{\boldsymbol{\alpha}}}
\newcommand{\bsbeta}{{\boldsymbol{\beta}}}
\newcommand{\bsgamma}{{\boldsymbol{\gamma}}}

\newcommand{\bstau}{\boldsymbol{\tau}}

\newcommand{\bsnu}{{\boldsymbol{\nu}}}

\newcommand{\bsk}{{\boldsymbol{k}}}

\newcommand{\bsm}{{\boldsymbol{m}}}
\newcommand{\bse}{{\boldsymbol{e}}}
\newcommand{\bsx}{{\boldsymbol{x}}}
\newcommand{\bsy}{{\boldsymbol{y}}}
\newcommand{\bsz}{{\boldsymbol{z}}}
\newcommand{\bsu}{{\boldsymbol{u}}}

\newcommand{\bhalf}{{\textstyle\boldsymbol{\frac{1}{2}}}}
\newcommand{\rd}{{\mathrm{d}}}
\newcommand{\bbR}{{\mathbb{R}}}
\newcommand{\bbA}{{\mathbb{A}}}
\newcommand{\bbB}{{\mathbb{B}}}

\newcommand{\bbN}{\mathbb{N}}

\newcommand{\calI}{\mathcal{I}}

\newcommand{\calO}{\mathcal{O}}
\newcommand{\calW}{\mathcal{W}}
\newcommand{\cL}{\mathcal{L}}
\newcommand{\setu}{\mathrm{\mathfrak{u}}}
\newcommand{\setv}{\mathrm{\mathfrak{v}}}

\newcommand{\KL}{Karhunen-Lo\`eve}
\newcommand{\indx}{{\mathfrak F}}

\newcommand{\mask}[1]{{}}
\newcommand{\bszero}{{\boldsymbol{0}}}

\newcommand{\calP}{\mathcal P}

\newcommand{\be}{\begin{equation}}
\newcommand{\ee}{\end{equation}}
\newcommand{\ba}{\begin{array}}
\newcommand{\ea}{\end{array}}
\newcommand{\beas}{\begin{eqnarray*}}
\newcommand{\eeas}{\end{eqnarray*}}
\newcommand{\bea}{\begin{eqnarray}}
\newcommand{\eea}{\end{eqnarray}}

\newcommand{\sell}{s_{\ell}}

\newcommand{\msell}{s_{\ell-1}}

\newcommand{\Ylangle}{{_{\cY'}}\langle}

\newcommand{\Yrangle}{\rangle_{\cY}}
\title{Multi-level higher order QMC Galerkin discretization
for \\
affine parametric operator equations}

\author{Josef Dick\footnotemark[2]
 \and Frances Y. Kuo\footnotemark[2]
 \and Quoc T. Le Gia\footnotemark[2]
 \and Christoph Schwab\footnotemark[3]}

\begin{document}

\maketitle

\renewcommand{\thefootnote}{\fnsymbol{footnote}}

 \footnotetext[2]{School of Mathematics and Statistics,
                  University of New South Wales, Sydney NSW 2052, Australia
                  ({\tt josef.dick@unsw.edu.au}, {\tt f.kuo@unsw.edu.au}, {\tt
                  qlegia@unsw.edu.au}). 
                  }
 \footnotetext[3]{Seminar for Applied Mathematics, ETH Z\"urich, ETH Zentrum,
                  HG G57.1, CH8092 Z\"urich, Switzerland ({\tt christoph.schwab@sam.math.ethz.ch}).
                  }

\renewcommand{\thefootnote}{\arabic{footnote}}

\begin{abstract}
We develop a convergence analysis of a multi-level algorithm combining
higher order quasi-Monte Carlo (QMC) quadratures with general
Petrov-Galerkin discretizations of countably affine parametric operator
equations of elliptic and parabolic type, extending both the multi-level
first order analysis in [\emph{F.Y.~Kuo, Ch.~Schwab, and I.H.~Sloan,
Multi-level quasi-Monte Carlo finite element methods for a class of
elliptic partial differential equations with random coefficient}
({Found.\ Comp.\ Math., 2015})] and the single level higher order
analysis in [\emph{J.~Dick, F.Y.~Kuo, Q.T.~Le~Gia, D.~Nuyens, and
Ch.~Schwab, Higher order QMC Galerkin discretization for parametric
operator equations} ({SIAM J.\ Numer.\ Anal., 2014})]. We cover, in
particular, both definite as well as indefinite, strongly elliptic systems
of partial differential equations (PDEs) in non-smooth domains, and
discuss in detail the impact of higher order derivatives of {\KL}
eigenfunctions in the parametrization of random PDE inputs on the
convergence results. Based on our \emph{a-priori} error bounds, concrete
choices of algorithm parameters are proposed in order to achieve a
prescribed accuracy under minimal computational work. Problem classes and
sufficient conditions on data are identified where multi-level higher
order QMC Petrov-Galerkin algorithms outperform the corresponding single
level versions of these algorithms. Numerical experiments confirm the theoretical results.
\end{abstract}



\pagestyle{myheadings} \thispagestyle{plain}

\section{Introduction}
\label{sec:Intro}

The efficient numerical computation of statistical quantities for
solutions of partial differential and of integral equations with random
inputs is a key task in uncertainly quantification and in the sciences. In
this paper, we combine the use of \emph{higher order quasi-Monte Carlo}
(QMC) \emph{quadrature} with \emph{Petrov-Galerkin discretization} in
a \emph{multi-level} algorithm to estimate a quantity of interest which
has been expressed as an \emph{infinite dimensional integral}. This paper
applies the new QMC theory developed in \cite{DKLNS13} (for a single level
algorithm) to the QMC Finite Element multi-level algorithm introduced in \cite{KSS13}, to
yield a potentially reduced exponent $a$ in the cost bound of
$\calO(\varepsilon^{-a})$, subject to a fixed error threshold
$\varepsilon>0$, with the constant implied in $\calO(\cdot)$ being independent
of the dimension of the integration domain.

The multi-level algorithm has first been introduced in \cite{H98} in the
context of integral equations and was independently rediscovered in
\cite{Gi08} in the context of simulation of stochastic differential
equations. A combination of the multi-level approach with the Monte Carlo
method has recently been developed for elliptic problems with random input
data in \cite{BSZ11,CGST11,CST13,HPS13,TSGU13,CHNST14}.

Let $\bsy := (y_j)_{j\ge 1}$ denote the possibly countable set of
parameters from a domain $U \subseteq \R^\N$, and let $A(\bsy)$ denote a
$\bsy$-parametric bounded linear operator between suitably defined spaces
$\bcX$ and $\bcY'$. We consider parametric operator equations: given $f\in
\bcY'$, for every $\bsy\in U$ find $u(\bsy)\in \bcX$ such that
\begin{equation}\label{eq:main}
  A(\bsy)\, u (\bsy) = f \;.
\end{equation}
Such parametric operator equations arise from partial differential
equations with random field input, see, e.g.,
\cite{SchwabGittelsonActNum11} and the references there. Following
\cite{ScMCQMC12,DKLNS13}, we consider in this paper problems where
$A(\bsy)$ has ``\emph{affine}'' \emph{parameter dependence}, i.e., there
exists a sequence $\{ A_j\}_{j\geq 0} \subset \cL(\bcX,\bcY')$ such that
for every $\bsy \in U$ we can write
\begin{equation}\label{eq:Baffine}
  A(\bsy) = A_0 + \sum_{j\ge 1} y_j\, A_j \;,
\end{equation}
and we restrict ourselves to the bounded (infinite-dimensional) parameter
domain
\[
 {U = [-\tfrac{1}{2},\tfrac{1}{2}]^\bbN}
\;.
\]
Some assumptions on the ``\emph{nominal}'' (or ``\emph{mean field}'')
operator $A_0$ and the ``\emph{fluctuation}'' operators $A_j$ are required
to ensure that the sum in \eqref{eq:Baffine} converges, and to ensure its
well-posedness, i.e., the existence and uniqueness of the parametric
solution $u(\bsy)$ in \eqref{eq:main} for all $\bsy\in U$; sufficient
conditions will be specified in \S\ref{sec:pre}. Further assumptions on
$A_0$ and $A_j$ are required for our regularity and approximation results;
these will also be given in \S\ref{sec:pre}. For now we mention only one
key assumption: there exists $\bar{t}\ge 0$ such that for every $0\le
t\le \bar{t}$ there exists a $0 < p_t < 1$ for which
\begin{equation} \label{eq:psumpsi0}
  \sum_{j\ge 1} \| A_j\|_{\cL(\bcX_t,\bcY'_t)}^{p_t} \,<\, \infty
  \qquad\mbox{and}\qquad
  \sum_{j\ge 1} \| A_j^*\|_{\cL(\bcY_t,\bcX'_t)}^{p_t} \,<\, \infty\;,
\end{equation}
where $\{\bcX_t\}_{t\ge 0}$ and $\{\bcY_t\}_{t\ge 0}$ denote scales of
smoothness spaces (see \eqref{eq:SmScal} ahead), with $\bcX_0 = \bcX$ and
$\bcY_0=\bcY$, and $\|\cdot\|_{\cL(\bcX_t,\bcY'_t)}$ denotes the operator
norm for the set of all bounded linear mappings from $\bcX_t$ to
$\bcY'_t$. As we will explain, it is natural to assume that $0<p_0\le
p_1\le\cdots\le p_{\bar{t}}< 1$. Assumption \eqref{eq:psumpsi0}
implies a decay of the fluctuations $A_j$ in \eqref{eq:Baffine}, with
stronger decay as the value of $p_0$ decreases.

For a quantity of interest (or ``goal'' functional) $G\in \bcX'$,
``ensemble averages'' of all possible realizations
of the operator equation
\eqref{eq:main} take the form of an integral over~$U$,
\begin{equation} \label{eq:int}
 I(G(u)) \,:=\, \int_U G(u(\bsy)) \,\rd\bsy\;.
\end{equation}
This calls for the consideration of QMC methods for numerical integration.
A \emph{single level} QMC strategy was developed and analyzed in
\cite{KSS12}, and subsequently generalized and improved in
\cite{ScMCQMC12,DKLNS13}. It contained three approximations: (i)
dimension-truncating the infinite sum in \eqref{eq:Baffine} to $s$ terms
(see \S\ref{sec:dimtrunc}), (ii) solving the corresponding operator
equation \eqref{eq:main} using a Finite Element method,
or more generally, Petrov-Galerkin discretization based on two dense,
one-parameter families $\{\cX^h\}_{h>0}\subset \cX$,
$\{\cY^h\}_{h>0}\subset \cY$ of finite dimensional
subspaces (see \S\ref{ssec:GalDisc}),
and (iii) approximating the corresponding integral \eqref{eq:int}
using a QMC rule with $N$ points in $s$ dimensions. Thus \eqref{eq:int}
was approximated by
\begin{equation} \label{eq:qmcG}
Q_{s,N}(G(u^h_s))
:=
\frac{1}{N} \sum_{n=0}^{N-1} G\big(u^h_s\big( {\bsy_n - {\bf \tfrac{1}{2}}}\big)\big)\;,
\end{equation}
where $\{\bsy_0,\ldots,\bsy_{N-1}\} \subset [0,1]^s$ are $N$ suitably
chosen QMC points, and the shift of coordinates by $\bhalf$ in
\eqref{eq:qmcG} accounts for the translation from $[0,1]^s$ to
$[-\frac{1}{2},\frac{1}{2}]^s$.

In \cite{KSS12}, first order QMC methods known as \emph{randomly shifted
lattice rules} were considered, together with first order finite
element methods, to achieve an overall root-mean-square error bound (with
respect to the random shift) of
\begin{equation} \label{eq:big1}
  \mbox{r.m.s.\ error} \,=\,
  \calO\left(s^{-2(1/p_0-1)} + N^{-\min(1/p_0-1/2,1-\delta)} + h^{t+t'}\right),
  \quad\delta>0\;,
\end{equation}
for a second order, elliptic PDE in the bounded spatial domain $D\subset \mathbb{R}^d$,
\begin{align} \label{eq:PDE1}
  &-\nabla \cdot \left(a(\bsy)\nabla u(\bsy)\right) \,=\, f\;, \quad
  u(\bsy)|_{\partial D} = 0\;,\quad
  a(\bsy) \,=\, a_0(\cdot) + \sum_{j\ge 1} y_j\,\psi_j(\cdot)\;,
\end{align}
which corresponds to the special case with $\cX=\cY=H^1_0(D)$, where
$0<p_0<1$, $0\le t,t'\le 1$, $f\in H^{-1+t}(D)$ and $G\in H^{-1+t'}(D)$.
The result is then generalized in \cite{ScMCQMC12} to the general affine
family of operator equations.
The implied constant in the bound \eqref{eq:big1}
and the QMC convergence rate with respect to $N$ are
independent of the integration dimension~$s$, and this is achieved by
choosing appropriate ``\emph{product and order dependent $($POD$)$
weights}'' in the function space setting for the QMC analysis. A suitable
\emph{generating vector} for the required lattice rule can be constructed
using a \emph{component-by-component} (CBC) algorithm, at a
(pre-computation) cost of $\calO(s\,N\log N + s^2 N)$ operations.

The QMC convergence rate in \eqref{eq:big1} was capped at order one in
\cite{KSS12,ScMCQMC12}, but this limitation was overcome in \cite{DKLNS13}
by considering a family of \emph{higher order digital nets} known as
(deterministic) \emph{interlaced polynomial lattice rules}, together with
higher order Galerkin discretization, to achieve an error bound of
\begin{equation} \label{eq:big2}
  \mbox{error} \,=\,
  \calO\left(s^{-2(1/p_0-1)} + N^{-1/p_0} + h^{t+t'} \right),
\end{equation}
for $0<p_0<1$, $0\le t,t'\le \bar{t}$, $f\in\cY_t'$ and $G\in\cX'_{t'}$.
The QMC convergence rate proved in \cite{DKLNS13} also gained an
additional factor of $N^{-1/2}$ as compared to the rate for
randomly shifted lattice rules in \cite{KSS12,ScMCQMC12}, thanks to a
new, non-Hilbert space setting for the QMC analysis (proposed already
in \cite{KSS11}). This approach is outlined in \S\ref{Sc:HiOrdQMC}.
The implied constant in \eqref{eq:big2} is again independent of $s$,
and this time it is achieved by choosing appropriate ``\emph{smoothness driven
product and order dependent $($SPOD$)$ weights}'' for the function space.
The generating vector for the required interlaced polynomial lattice rule
can again be constructed using a CBC algorithm, at a slightly higher cost
of $\calO(\alpha\,s\,N\log N + \alpha^2\,s^2 N)$ operations, with $\alpha
= \lfloor 1/p_0\rfloor + 1\ge 2$.

To reduce the computational cost required to achieve the same error, a
novel \emph{multi-level} algorithm was introduced and analyzed in
\cite{KSS13}. It takes the form
\begin{equation}\label{eq:QL*}
 Q^L_*(G(u)) \,:=\,
 \sum_{\ell=0}^L Q_{s_\ell,N_\ell}(G(u^{h_\ell}_{s_\ell} - u^{h_{\ell-1}}_{s_{\ell-1}}))\;,
\end{equation}
where each $Q_{s_\ell,N_\ell}$ is a randomly shifted lattice rule with
$N_\ell$ points in $s_\ell$ dimensions, and where
$u^{h_{-1}}_{s_{-1}}:=0$. The corresponding root-mean-square error bound
is
\begin{align} \label{eq:big3}
  &{\mbox{r.m.s.\ error} \,=\,} \nonumber \\
  &\calO\left(
  s_L^{-2(1/p_0-1)} + h_L^{t+t'}
  + \sum_{\ell=0}^L N_\ell^{-\min(1/p_1-1/2,1-\delta)} \left(
  s_{{\ell-1}}^{-(1/p_0 - 1/p_1)} + h_{{\ell-1}}^{t+t'} \right) \right),
  \quad\delta>0,
\end{align}
where $s_{-1}:=1$, $h_{-1}:=1$, $0<p_0 \le p_1 < 1$, $0\le t,t'\le 1$, and
the implied constant is independent of $s$, with appropriately chosen POD
weights. Assuming that the overall cost of \eqref{eq:QL*} is
$\calO(\sum_{\ell=0}^L s_\ell N_\ell h_\ell^{-d})$, an argument based on
the Lagrange multipliers was used to optimize the choice of $s_\ell$ and
$N_\ell$ in relation to $h_\ell\asymp 2^{-\ell}$. Note that the QMC
convergence rate with respect to $N_{\ell}$ in \eqref{eq:big3} depends on
$p_1$, rather than on $p_0$.

In this paper, we replace the randomly shifted lattice rules in
\eqref{eq:QL*} by interlaced polynomial lattice rules as in
\cite{DKLNS13}, to achieve the improved error bound
\begin{equation} \label{eq:big4}
 \mbox{error} \,=\,
  \calO\left(
  s_L^{-2(1/p_0-1)} + h_L^{t+t'} + \sum_{\ell=0}^L N_\ell^{-1/p_t}
  \left( s_{\ell-1}^{-(1/p_0-1/p_t)} + h_{\ell-1}^{t+t'} \right) \right),
\end{equation}
where $0<p_0 \le p_t < 1$, $0\le t,t'\le \bar{t}$. The implied constant is
independent of $s$, again, under the provision of appropriate SPOD
weights. Comparing \eqref{eq:big4} with \eqref{eq:big3}, we see that the
convergence rate is no longer capped at order one as expected, and there
is a gain of the additional factor $N_\ell^{-1/2}$ as in \eqref{eq:big2}.
However, the convergence rate depends now on the summability exponent
$p_t$ rather than $p_0$ or $p_1$.

As we argue in \S \ref{sec:CovOp} of
this paper,
in many examples, the 
exponent $p_t$ in \eqref{eq:psumpsi0} satisfies
\begin{equation} \label{eq:pt}
p_t =  \frac{p_0}{1- t p_0/d }, \quad 1\leq t \leq \bar{t}\;,
\end{equation}
which could be much larger than $p_0$. The requirement $p_t<1$ imposes a
constraint on $\bar{t}$, the maximum allowable value of $t$ and $t'$,
which in turn reduces the convergence rate in \eqref{eq:big4}. In some
scenarios the potential gain of the multi-level algorithm \eqref{eq:QL*}
over the single level algorithm \eqref{eq:qmcG} (whose error bound depends
only on $p_0$) can be limited.

The outline of this paper is as follows. In \S \ref{sec:pre}, we formulate
the affine parametric operator equations, specify all assumptions
which are subsequently needed in our QMC error analysis, and introduce an
abstract Petrov-Galerkin discretization of these operator equations which
covers most Galerkin discretizations of parabolic and elliptic partial
differential equations in a bounded spatial domain $D$. Examples include
second order, elliptic divergence form PDEs in polyhedral domains as
considered in \cite{NisSc13}. We elaborate on \eqref{eq:pt} resulting from
random field modelling with covariance operators chosen as negative powers
of second order, elliptic pseudo-differential operators in~$D$. We also
give in \S\ref{sec:pre} a synopsis of the key results of our single level
QMC Petrov-Galerkin error analysis in \cite{DKLNS13}, to the extent
required for the present work. In \S\ref{sec:MultAlg}, we
introduce the multi-level 
QMC Petrov-Galerkin approximation as direct generalization of the
multi-level algorithm based on (first order) randomly shifted
lattice rules analyzed in \cite{KSS13}.
We present the basic error bounds
for the combined QMC Petrov-Galerkin error, refining and extending
the analysis of \cite{DKLNS13}, and derive concrete selections of the
algorithm parameters based on optimization of the error bounds. The
proposed parameter choices are then used to derive asymptotic accuracy
versus work bounds for the proposed algorithms, subject to given data
regularity in terms of spatial differentiability as well as decay of the
covariance spectrum of the random field input.
Finally in \S\ref{sec:ConclGen} we give some concluding remarks.
\section{Problem formulation}
\label{sec:pre}
Generalizing results of \cite{CDS1}, we study well-posedness, regularity
and polynomial approximation of solutions for a family of abstract
parametric saddle point problems, with operators depending on a sequence
of parameters. The results cover a wide range of affine parametric
operator equations: among them are stationary and time-dependent diffusion
in random media \cite{CDS1}, wave propagation \cite{HoSc12Multi}, and
optimal control problems for uncertain systems \cite{KunothCS2011}.
%
\subsection{Affine parametric operator equations}
\label{ssec:affparops}
We denote by $\bcX$ and $\bcY$ two separable and reflexive Banach spaces
over $\mathbb{R}$ (all results will hold with the obvious modifications
also for spaces over $\mathbb{C}$) with (topological) duals $\bcX'$ and
$\bcY'$, respectively. By $\cL(\bcX,\bcY')$, we denote the set of bounded
linear operators $A:\bcX \to\bcY'$.

%
%
%
%
A particular instance of \eqref{eq:main} and \eqref{eq:Baffine} are boundary
value problems of second order, elliptic (systems of) partial differential
equations such as linear elasticity in anisotropic, parametric medium.
Here, $\bcX = \bcY = H^1_0(D)^\iota$ with $\iota\ge 1$, and
$A(\bsy)$ is given by 
the divergence-form elliptic differential operator which acts on vector
functions $u(\bsy) : D\mapsto \bbR^\iota$ via
\begin{equation}\label{eq:PDE2}
  (A(\bsy)u(\bsy))_l
  \,=\, - \sum_{i,j=1}^d \sum_{k=1}^{\iota}
  \partial_i (a^{ij}_{kl}(\bsx,\bsy) \partial_j u_k(\bsx,\bsy))
  \,=\,  f_l
 \quad\mbox{in}\quad D,\; l=1,...,\iota,
\end{equation}
and $u(\bsy)|_{\partial D} = 0$.
%
In the scalar, isotropic case of \eqref{eq:PDE2} which was considered in
\cite{KSS12}, we have $\iota = 1$ and the coefficient function
$a^{ij}(\bsy) = \delta_{ij} a(\bsy)$ with $a(\bsy)$ as in \eqref{eq:PDE1}.
%
%
For linearized elasticity, $\iota = d$ in \eqref{eq:PDE2}. Other boundary
conditions in \eqref{eq:PDE2} could equally well be considered (we refer
to \cite[Sec.1.2]{NisSc13} for details).

As we explained in the introduction, let $\bsy := (y_j)_{j \geq 1} \in U =
[-\frac{1}{2},\frac{1}{2}]^\bbN$ be a countable set of parameters.
For every $f\in \bcY'$ and for every $\bsy\in U$, we solve the parametric
operator equation \eqref{eq:main}, where the operator
$A(\bsy)\in\cL(\bcX,\bcY')$ is of affine parameter dependence, see
\eqref{eq:Baffine}. We associate with the operators $A_j$ the
parametric bilinear
forms $\fa_j(\cdot,\cdot):\bcX\times \bcY \rightarrow \mathbb{R}$ via
$$
  \forall v\in \cX,\;w\in \cY:\quad
  \fa_j(v,w) \,=\, {_{\cY'}}\langle  A_j v, w \rangle_{\cY}\;,
  \quad j=0,1,2,\ldots
  \;.
$$
Similarly, for $\bsy\in U$ we associate with the
parametric operator $A(\bsy)$ the
parametric bilinear form $\fa(\bsy;\cdot,\cdot): \bcX\times\bcY\to\R$ via
\[ 
  \forall v\in \cX,\;w\in \cY:\quad
  \fa(\bsy;v,w) \,=\, {_{\cY'}}\langle A(\bsy) v, w\rangle_{\cY}\;.
\] 

In order for the sum in \eqref{eq:Baffine} to converge, we impose the
assumptions below on the sequence $\{A_j\}_{j\geq 0}\subset
\cL(\cX,\cY')$.

\begin{assumption}\label{ass:AssBj}
The sequence $\{ A_j \}_{j\geq 0}\subset \cL(\bcX,\bcY')$ in
\eqref{eq:Baffine} satisfies: 
\begin{enumerate}
\item%
$A_0\in \cL(\bcX,\bcY')$ is boundedly invertible, i.e., there exists a
constant $\mu_0 > 0$ such that
\begin{equation}\label{eq:B0infsup} 
 \inf_{0\ne v \in \bcX} \sup_{0\ne w \in \bcY}
 \frac{\fa_0(v,w)}{\| v \|_{\bcX} \|w\|_{\bcY}}
 \ge \mu_0\;,\quad
 \inf_{0\ne w \in \bcY} \sup_{0\ne v \in \bcX}
 \frac{\fa_0(v,w)}{\| v \|_{\bcX} \|w\|_{\bcY}}
 \ge \mu_0
 \;.
\end{equation}
\item%
The \emph{fluctuation operators} $\{ A_j \}_{j\geq 1}$ are small with
respect to $A_0$ in the following sense:
there exists a constant $0 < \kappa < 2$ such that
\begin{equation} \label{eq:Bjsmall} 
 \sum_{j\geq 1} \beta_{0,j} \leq \kappa < 2 \;,
 \quad\mbox{where}\quad
 \beta_{0,j} \,:=\, \| A_0^{-1} A_j \|_{\cL(\cX,\cX)}\;,
 \quad j=1,2,\ldots
 \;.
\end{equation}
\end{enumerate}
\end{assumption}
%
%
\begin{theorem}[{cp.~\cite[Theorem 2]{ScMCQMC12}}]
\label{thm:BsigmaInv} Under Assumption~\ref{ass:AssBj}, for every
realization $\bsy\in \cS$ of the parameter vector, the affine parametric
operator $A(\bsy)$ given by \eqref{eq:Baffine} is boundedly invertible,
uniformly with respect to $\bsy$.
In particular, for every $f \in \bcY'$ and for every $\bsy \in \cS$, the
parametric operator equation
\begin{equation} \label{eq:parmOpEq}
 \mbox{find} \quad u(\bsy) \in \bcX:\quad
 \fa(\bsy;u(\bsy), w) \,=\,  {_{\bcY'}}\langle  f,w \rangle_{\bcY}
 \quad
 \forall w \in \bcY
\end{equation}
admits a unique solution $u(\bsy)$ which satisfies
the a-priori estimate
\[ 
 \| u(\bsy) \|_{\bcX}
 \,\leq\,
 \frac{1}{\mu}
 \, \| f \|_{\bcY'}\;, \quad
  \mbox{with}\quad \mu = (1 - \kappa/2)\,\mu_0
\;.
\] 
\end{theorem}
%
%
\subsection{Parametric and spatial regularity of solutions}
\label{ssec:anadepsol}
First we establish the regularity of the solution $u(\bsy)$ of the
parametric, variational problem \eqref{eq:parmOpEq} with respect to the
parameter vector $\bsy$. In the following, let $\N_0^\N$ denote the set of
sequences $\bsnu = (\nu_j)_{j\geq 1}$ of non-negative integers $\nu_j$, and
let $|\bsnu| := \sum_{j\geq 1} \nu_j$. For $|\bsnu|<\infty$, we denote the
partial derivative of order $\bsnu$ of $u(\bsy)$ with respect to $\bsy$ by
\[
  \partial^\bsnu_\bsy u(\bsy) \,:=\,
  \frac{\partial^{|\bsnu|}}{\partial^{\nu_1}_{y_1}\partial^{\nu_2}_{y_2}\cdots}u(\bsy)
  \;.
\]

\begin{theorem}[cp.~\cite{CDS1,KunothCS2011}]\label{thm:Dsibound}
Under Assumption~\ref{ass:AssBj}, there exists a constant $C_0 > 0$ such
that for every $f\in \bcY'$ and for every $\bsy\in \cS$, the partial
derivatives of the parametric solution $u(\bsy)$ of the parametric
operator equation \eqref{eq:main} with affine parametric, linear operator
\eqref{eq:Baffine} satisfy the bounds
\begin{equation} \label{eq:Dsibound}
\|\partial^\bsnu_\bsy u(\bsy)\|_\bcX
\,\le\,
C_0\, |\bsnu|! \,\bsbeta_0^\bsnu\, \| f\|_{\bcY'}
\quad \mbox{for all } \bsnu \in \N_0^\N \mbox{ with } |\bsnu|<\infty
\;,
\end{equation}
where $0! :=1$, $\bsbeta_0^\bsnu := \prod_{j\ge 1} \beta_{0,j}^{\nu_j}$,
with $\beta_{0,j}$ as in \eqref{eq:Bjsmall}, and $|\bsnu| = \sum_{j \ge 1}
\nu_j$.
\end{theorem}

\medskip
%
For the spatial regularity, we assume given {\em scales of smoothness
spaces} $\{ \cX_t \}_{t \geq 0}$, $\{ \cY_t \}_{t\geq 0}$, with
\begin{equation}\label{eq:SmScal}
\begin{aligned}
 \cX &= \cX_0 \supset \cX_1 \supset \cX_2 \supset \cdots\;,
 &\cY &= \cY_0 \supset \cY_1 \supset \cY_2 \supset \cdots\;,
 \quad\mbox{and}
 \\
 \cX' &= \cX'_0 \supset \cX'_1 \supset \cX'_2 \supset \cdots\;,
 &\cY' &= \cY'_0 \supset \cY'_1 \supset \cY'_2 \supset \cdots
 \;.
\end{aligned}
\end{equation}
The scales 
are assumed to be defined also for non-integer values of the smoothness
parameter $t\geq 0$ by interpolation. For self-adjoint operators, usually
$\cX_t  = \cY_t$. For example, in diffusion problems in {\em convex
domains} $D$ considered in \cite{CDS1,KSS12}, the smoothness scales
\eqref{eq:SmScal} are $\cX = \cY = H^1_0(D)$, $\cX_1 = \cY_1 = (H^2\cap
H^1_0)(D)$, $\cY' = H^{-1}(D)$, $\cY'_1 = L^2(D)$. In a non-convex polygon
(or polyhedron), analogous smoothness scales are available, but involve
Sobolev spaces with weights.
%
In \cite{NisSc13}, this kind of abstract regularity result was established
for a wide range of second order parametric, elliptic systems in 2D and
3D, also for higher order regularity. The smoothness scales $\{ \bcX_t
\}_{t\geq 0}$ and $\{ \bcY'_t\}_{t\geq 0}$ are then weighted Sobolev
spaces $\cK^{t+1}_{a+1}(D)$ of Kondratiev type in $D$, and $\cX_t =
\cK^{t+1}_{a+1}(D)$, $\cY'_t = \cK^{t-1}_{a-1}(D)$ in this case. The
Finite Element spaces which realize the maximal convergence rates (beyond
order one) are regular, simplicial families in the sense of Ciarlet, on
suitably refined meshes which compensate for the corner and edge
singularities.

The maximum amount of smoothness in the scale $\cX_t$, denoted by
$\bar{t}\ge 0$, depends on the problem class under consideration and
on the Sobolev scale: e.g., for elliptic problems in polygonal domains, it
is well known that choosing for $\cX_t$ the usual Sobolev spaces will
allow \eqref{eq:Regul} with $t$ only in a possibly small interval $0< t
\leq \bar{t}$, whereas choosing $\cX_t$ as Sobolev spaces with weights
will allow rather large values of $\bar{t}$ (see, e.g., \cite{NisSc13}).

We next formalize the parametric regularity hypothesis.
\begin{assumption}\label{ass:XtYt}
There exists $\bar{t}\ge 0$ such that the following conditions hold:
\begin{enumerate}
\item For every $t,t'$ satisfying $0\le t,t'  \le \bar{t}$,
we have
\begin{equation}\label{eq:Regul}
 \sup_{\bsy\in \cS} \| A(\bsy)^{-1} \|_{\cL(\cY'_t, \cX_t)} < \infty
 \quad\mbox{and}\quad
 \sup_{\bsy\in \cS} \| (A^*(\bsy))^{-1} \|_{\cL(\cX'_{t'}, \cY_{t'})}  < \infty\;.
\end{equation}
%
Moreover, for every $t$ satisfying $0\le t\le\bar{t}$, there
exist summability exponents $0\le p_0 \le p_t \le p_{\bar{t}}<1$ such
that
\begin{equation} \tag{\ref{eq:psumpsi0}} 
 \sum_{j \ge 1} \| A_j \|^{p_t}_{\cL(\cX_t,\cY'_t)} < \infty
\;.
\end{equation}
%
%
\item
    Let $\bsu(\bsy) = (A(\bsy))^{-1}f$ and
    $w(\bsy) = (A^*(\bsy))^{-1}G$.
    For $0\le t,t'\le\bar{t}$,
    there exist constants $C_t,C_{t'}>0$
    such that for every  $f\in \cY'_t$ and $G\in\cX'_{t'}$ holds
\begin{equation}\label{ass:a priory}
  \sup_{\bsy\in U} \|u(\bsy)\|_{\cX_t} \le C_t \|f\|_{\cY'_t}
  \quad\mbox{and}\quad
 \sup_{\bsy\in U} \|w(\bsy)\|_{\cY_{t'}} \le C_{t'} \|G\|_{\cX'_{t'}}\;.
\end{equation}
Moreover, for every $t$ satisfying $0\le t\le\bar{t}$, there
exists a sequence $\bsbeta_t = (\beta_{t,j})_{j\geq 1} \in
\ell^{p_t}(\N)$, i.e., satisfying
\begin{equation} \label{eq:AssBjp}
 \sum_{j\ge 1} \beta_{t,j}^{p_t} \,<\,\infty\;,
\end{equation}
such that for
every $0\leq t,t' \leq \bar{t}$ and for every
$\bsnu\in\bbN_0^\bbN$ with $|\bsnu|<\infty$
we have
\begin{align}\label{ass:diffy}
 \sup_{\bsy \in U} \|\partial^{\bsnu}_{\bsy} u (\bsy) \|_{\cX_t}
 &\,\le\, C_t\, |\bsnu|!\, \bsbeta_t^{\bsnu}\, \|f\|_{\cY_t'},
 \\
 \label{eq:refAd}
 \sup_{\bsy\in U} \| \partial^{\bsnu}_{\bsy} w(\bsy)\|_{\bcY_{t'}}
 &\,\le\, C_{t'}\, |\bsnu|!\, \bsbeta_{t'}^\bsnu\, \| G \|_{\bcX'_{t'}}
\;.
\end{align}
\item The operators $A_j$ are enumerated so that the sequence
    $\bsbeta_0$ in \eqref{eq:Bjsmall} satisfies
\begin{equation} \label{eq:ordered}
  \beta_{0,1} \ge \beta_{0,2} \ge \cdots \ge \beta_{0,j} \ge \, \cdots\;.
\end{equation}
\end{enumerate}
\end{assumption}

Parametric regularity as in Item~2 of Assumption \ref{ass:XtYt} is
available for numerous parametric differential equations (see
\cite{SchwabGittelsonActNum11,HoaSc12Wave,HaSc11,KunothCS2011} and the
references there) as well as for posterior densities in Bayesian inverse
problems with uniform priors (see, e.g., \cite{SS12,SS13} and the
references there). Writing $A(\bsy) = A_0(I + \sum_{j\geq 1}y_j
A_0^{-1}A_j)$, a Neumann series argument shows that a sufficient condition
for \eqref{eq:Regul} to hold is $A_0^{-1}\in \cL(\bcY'_t,\bcX_t)$, $A_j
\in \cL(\bcX_t,\bcY'_t)$ and that
\[
  \sum_{j\ge1} \| A_0^{-1}A_j \|_{\cL(\bcX_t,\bcX_t)} < 2
\;.
\]
We may estimate
\[ 
\| A_0^{-1} A_j \|_{\cL(\bcX_t,\bcX_t)}
\leq
\| A_0^{-1} \|_{\cL(\bcY'_t, \bcX_t)} \| A_j \|_{\cL(\bcX_t,\bcY'_t)}\;,
\;\; j=1,2,3,\cdots\;,
\] 
and since $A_j = A_0 A_0^{-1} A_j$ we have
$ \|A_j\|_{\cL(\cX_t,\cY'_t)}
  \le \| A_0 \|_{\cL(\cX_t,\cY_t')}
  \|A_0^{-1} A_j\|_{\cL(\bcX_t,\bcX_t)}$.
Combining these two estimates, we have for every $j\geq 1$
\begin{equation}\label{eq:Aknorm}
  \| A_0 \|_{\cL(\cX_t,\cY_t')}^{-1}
  \le \frac{\| A_0^{-1} A_j \|_{\cL(\bcX_t,\bcX_t)}}{ \| A_j \|_{\cL(\bcX_t,\bcY'_t)} }
  \le \| A_0^{-1} \|_{\cL(\bcY'_t, \bcX_t)}\;.
\end{equation}
This shows that condition \eqref{eq:psumpsi0} is equivalent to (but not
identical to) the condition that $\sum_{j \ge 1} \| A_0^{-1} A_j
\|^{p_t}_{\cL(\cX_t,\cX_t)} < \infty$.

\subsection{Illustration of Assumption~\ref{ass:XtYt}}
\label{sec:CovOp}
The condition \eqref{eq:B0infsup} of Assumption~\ref{ass:AssBj}
implies $A_0^{-1}\in \cL(\cY',\cX)$ so that for every $\bsy \in U$ we have
$A(\bsy)\,u(\bsy) = f \Longleftrightarrow B(\bsy)\,u(\bsy) = \tilde{f}$,
where $B(\bsy) := I + \sum_{j\geq 1}y_j (A_0^{-1}A_j)$ and $\tilde{f}:=
A_0^{-1} f$. Taking $\bsy = \bszero$ in \eqref{eq:Regul} yields $A_0^{-1}
\in \cL(\cY'_t,\cX_t)$, while \eqref{eq:psumpsi0} and \eqref{eq:Aknorm}
together gives $A_0^{-1}A_j \in \cL(\cX_t,\cX_t)$ for $j=1,2,\ldots$.
Hence \eqref{eq:AssBjp} holds with $\beta_{t,j} := \| A_0^{-1}A_j
\|_{\cL(\cX_t,\cX_t)}$. We may now apply the argument in \cite{CDS1} to
the affine parametric operator equation $B(\bsy)u(\bsy)=\tilde{f}$ to
obtain \eqref{ass:diffy}. Repeating this argument for the adjoint equation
$B(\bsy)^*w(\bsy) = \tilde{G}:=A_0^{-*}G \in \cX_t$ then yields
\eqref{eq:refAd}.

%

The summability \eqref{eq:AssBjp} is well known to be related to the
smoothness of the covariance kernels of the random coefficient; see e.g.,
\cite[Appendix]{ST06} for details. We illustrate \eqref{eq:AssBjp} in the
context of the scalar, parametric diffusion problem \eqref{eq:PDE1}. One
source of the $\psi_j$ in \eqref{eq:PDE1} are principal component analysis
expansions such as {\KL} expansions of random coefficients, and therefore
\eqref{eq:AssBjp} is a sparsity assumption on the coefficient function
sequence $\{\psi_j\}_{j\geq 1}$ and their derivatives of orders
$t=1,2,\ldots, \lfloor\bar{t}\rfloor$.

%
Consider the Dirichlet Laplacean $-\Delta_d$ in the unit cube $D=(0,1)^d$
with $d\ge 1$. This is an unbounded, self-adjoint operator on $L^2(D)$
with a discrete spectrum consisting of countably many real eigenvalues
which accumulate only at infinity. It is elementary to verify by
separation of variables that the eigenpairs of $-\Delta_d$ are
$$
 -\Delta_d\,\tilde{\psi}_{\bsk}
 \,=\,
 \lambda_{\bsk}\, \tilde{\psi}_{\bsk} \quad\mbox{in}\quad D,
 \quad
 \tilde{\psi}_{\bsk}|_{\partial D} = 0\;,
 \quad \bsk = (k_1,\ldots,k_d) \in \bbN^d
 \;,
$$
with
\begin{equation}\label{eq:sinuseig}
 \lambda_{\bsk} \,=\, \pi^2 (k_1^2+\cdots+k_d^2),\;\;
 \tilde{\psi}_{\bsk}(\bsx) \,=\, \prod_{i=1}^d \sin(\pi k_i x_i)
\;.
\end{equation}
Enumerating $\{\lambda_{\bsk}\}_{\bsk\in \bbN^d}$ in non-decreasing order
$\{\lambda_j\}_{j\ge 1}$, there hold the {\em Weyl asymptotics} (see,
e.g., \cite{Shubin} and the references there)
\begin{equation}\label{eq:Weyl}
\lambda_j \sim j^{2/d} \quad\mbox{as}\quad j\to\infty \;.
\end{equation}

Next, we consider again the domain $D$, but now for some real parameter
$\theta > 0$ the {\em Covariance operator} $\cC_\theta =
(-\Delta_d)^{-\theta}$. Then, for any $\theta > 0$, $\cC_\theta \in
\cL(L^2(D),L^2(D))$ is a compact, self-adjoint operator whose spectrum
$\sigma(\cC_\theta) = (\mu_j)_{j\geq 1}$ consists of countably many, real
eigenvalues which we enumerate again in non-increasing order. By the
spectral mapping theorem and the Weyl asymptotics \eqref{eq:Weyl}, the
operators $\cC_\theta$ have the same eigenfunctions $\tilde{\psi}_j$ as
the operator $-\Delta_d$, and the corresponding eigenvalues $\mu_j$ of
$\cC_\theta$ have the asymptotics
\[ 
\mu_j \sim j^{-2\theta/d} \quad\mbox{as} \quad j\to \infty \;.
\] 
In {\KL} expansions with uncertain coefficients, we have \eqref{eq:PDE1}
with $\psi_j:=\sqrt{\mu_j}\, \tilde{\psi}_j$. Clearly in this case we have
$\|\tilde\psi_j\|_{L^\infty(D)}\le 1$ for $j\ge 1$, which yields $\|\psi_j
\|_{L^\infty(D)} \lesssim j^{-\theta/d}$, from which we conclude that
\[
  \sum_{j\geq 1} \|  \psi_j \|_{L^\infty(D)}^{p_0} \,<\,\infty
  \quad\mbox{with}\quad
  p_0 \,>\, \frac{d}{\theta}\;.
\]
We find for $t=0,1,2,...$ and for every $j\in \bbN$ that $\| \tilde\psi_j
\|_{W^{t,\infty}(D)} \lesssim j^{{t/d}}$,
and therefore
\begin{equation}\label{eq:Wtinftybd}
\| \psi_j \|_{W^{t,\infty}(D)} \lesssim j^{{(t-\theta)/d}}\;,
\end{equation}
with the implied constant depending on $t$, but independent of $j\in
\bbN$. So it holds
\[ 
 \sum_{j\geq 1} \|  \psi_j \|_{W^{t,\infty}(D)}^{p_t} <\infty\;,
 \quad\mbox{with}\quad
 p_t \,:=\, \frac{p_0}{1-tp_0/d} \,<\, \frac{d}{\theta- t } \;.
\] 
The requirement that $p_t < 1$ means
\[ 
 \bar{t} =
d \left( \frac{1}{p_0}-1 \right)
<
\theta - d 
\;.
\] 
Thus sparsity of expansions of higher order $t$ is only available for
sufficiently large $\theta > 0$, at least in this example where
\eqref{eq:Wtinftybd} is sharp.

The preceding arguments rely strongly on the explicit formulas
\eqref{eq:sinuseig}. For covariance operators of the form
$\cC = B^{-\theta}$ for a general, positive and second order, self-adjoint
elliptic divergence form partial differential operator $B\in
\cL(\bcX,\bcX')$ with non-constant, H\"older regular coefficients in a
polygonal/polyhedral domain $D$, the spectral asymptotics of the
$\lambda_j$ as $j\to \infty$ is well known to hold as well (see e.g.\
\cite[Theorem~15.2]{Shubin} for smooth domains and smooth coefficients,
and \cite{Miyazaki09} for elliptic, divergence-form operators with
non-smooth coefficients). Importantly, also in this case, the
eigenfunctions $\tilde{\psi}_j$ are bounded, but may exhibit singularities
at corners and edges of the domain $D$, so that they belong only to {\em
weighted $W^{t,\infty}(D)$ spaces} denoted in \cite{NisSc13} by
$\calW^{t,\infty}(D)$; coefficients in such spaces for
\eqref{eq:PDE2} are admissible in the results of \cite{NisSc13}, cp.\
\cite[Eq. (2.3)]{NisSc13}, where also conditions \eqref{ass:diffy} and
\eqref{eq:refAd} have been verified for parametric, elliptic systems
\eqref{eq:PDE2}. In the context of the parametric, second-order, elliptic
divergence-form PDE \eqref{eq:PDE2}, we have
$\| A_j \|_{\cL(\bcX_t,\bcY_t')} \lesssim \| \psi_j
\|_{\calW^{j,\infty}(D)} $ (cp.\ \cite[Eq.\ (2.6)]{NisSc13}) and $\|
A_0^{-1} \|_{\cL(\bcY'_t, \bcX_t)}$ being bounded in a scale of weighted
Sobolev spaces (cp.\ \cite[Corollary 2.1]{NisSc13} with the identification
$\bcX_t = {\cal K}^{1+t}_{a+1}(D)$), with $\lesssim$ denoting an absolute
constant (depending on $t$, but not on $j$); 
we refer to \cite[Eq.(2.6)]{NisSc13} for details.
%
%
\subsection{Petrov-Galerkin discretization}
\label{ssec:GalDisc}
Since the exact solution is not available explicitly, we will have to
compute, for given $\bsy\in \cS$, an approximate solution obtained by {\em
Petrov-Galerkin discretization}.

\begin{theorem}[{cp.~\cite[\S2.4]{DKLNS13}}] \label{thm:Galerkin}
Let $\{ \cX^h \}_{h>0}\subset\cX$ and $\{ \cY^h \}_{h>0}\subset\cY$ be two
families of finite dimensional subspaces which are dense in $\cX$ and in
$\cY$, respectively. Assume moreover the approximation property and that
the Petrov-Galerkin subspace pairs $\cX^h\times \cY^h$ are inf-sup stable
with respect to the nominal bilinear form $\fa_0(\cdot,\cdot)$, as in
\eqref{eq:B0infsup}, with constant $\bar{\mu}_0>0$ independent of $h$.
This implies the discrete inf-sup conditions for the bilinear form
$\fa(\bsy;\cdot,\cdot)$, uniformly with respect to $\bsy\in U$, with
constant $\bar{\mu} \,=\, (1-\kappa/2)\, \bar{\mu}_0>0$.

Then for every $\bsy\in U$ we have existence, uniqueness and
\textnormal{(}uniform with respect to $\bsy$\textnormal{)}
quasioptimality of the Petrov-Galerkin solutions,
ie., 
for every $0<h \leq h_0$ and for every $\bsy \in \cS$,
the Petrov-Galerkin approximations $u^h(\bsy)\in\cX^h$, given by
\begin{equation} \label{eq:parmOpEqh}
\mbox{find} \; u^h(\bsy) \in \bcX^h :
\quad
\fa(\bsy;u^h(\bsy),w^h) = 
{_{\cY'}} \langle f, w^h \rangle_{\cY}
\quad
\forall w^h\in \bcY^h\;,
\end{equation}
are well defined, and stable, i.e.,
they satisfy the uniform a-priori estimate
\begin{equation}\label{eq:FEstab}
 \| u^h(\bsy) \|_{\bcX} \,\le\, \frac{1}{\bar{\mu}}\, \| f \|_{\bcY'}
\;.
\end{equation}
Moreover, for $0<t\le \bar{t}$, if the basis functions have smoothness
degree $\lceil t\rceil$ then there exists a constant $C_t>0$ such that
for every $\bsy\in U$
\begin{equation} \label{eq:reg1}
 \| u(\bsy) - u^h(\bsy) \|_{\cX} \,\le\, C_t\, h^t\, \| u(\bsy)\|_{\cX_t}
\;.
\end{equation}

Additionally, we assume uniform inf-sup stability of the pairs
$\bcX^h\times \bcY^h$ for the adjoint problem, so that for $0<t'\le
\bar{t}$ there exists a constant $C_{t'}>0$ such that for all $0<h\le h_0$
and $\bsy\in U$,
\begin{equation} \label{eq:reg2}
 \| w(\bsy) - w^h(\bsy) \|_{\cY}
 \,\le\, C_{t'} h^{t'} \| w(\bsy)\|_{\cY_{t'}}
\;.
\end{equation}
Then, for every $f\in \cY'_{t}$ and $G\in \cX'_{t'}$ with $0<t,t'\leq
\bar{t}$ and for every $\bsy\in \cS$, as $h\to 0$, there exists a constant
$C>0$ independent of $h>0$ and of $\bsy\in U$ such that the Galerkin
approximations $G(u^h(\bsy))$ satisfy
\begin{align} \label{eq:Gconvest}
  \left| G(u(\bsy)) - G(u^h(\bsy)) \right|
  &\,\le\, C\, h^{t+t'} \, \| f \|_{\cY'_{t}} \, \| G \|_{\cX'_{t'}}
\;.
\end{align}
\end{theorem}
\subsection{Dimension truncation}
\label{sec:dimtrunc}
We truncate the infinite sum in \eqref{eq:Baffine} to $s$ terms and
solve the corresponding operator equation \eqref{eq:main}
approximately using Galerkin discretization from
two dense, one-parameter families
$\{\cX^h\}\subset \cX$, $\{\cY^h\}\subset \cY$ of subspaces
of $\cX$ and $\cY$:
for $s\in \bbN$ and $\bsy\in U$, we define
\begin{equation}\label{equ:defAs}
\fa_{s}(\bsy;v,w) := \Ylangle A^{(s)}(\bsy) v, w\Yrangle,
\quad\text{with}\quad
  A^{(s)}(\bsy) :=  A_0 + \sum_{j=1}^{s} y_j A_j.
\end{equation}
Then, for every $0<h \le h_0$ and every $\bsy\in U$,
the dimension-truncated Galerkin solution
$u^h_s(\bsy)$ is the solution of
\begin{equation} \label{eq:uhs}
\text{find } u^h_s(\bsy) \in \cX^h: \quad
\fa_s(\bsy;u^h_s(\bsy),w^h) = \Ylangle f, w^h \Yrangle \quad
\forall w^h\in \cY^h.
\end{equation}
By choosing $\bsy = (y_1,\ldots,y_s,0,0,\ldots)$,
Theorem~\ref{thm:Galerkin} remains 
valid for the dimensionally truncated problem \eqref{eq:uhs}, and hence
\eqref{eq:FEstab} holds with $u^h_s(\bsy)$ in place of $u^h(\bsy)$.

%

\begin{theorem}[{cp.~\cite[Theorem~2.6]{DKLNS13}}] \label{thm:trunc}
Under Assumption~\ref{ass:AssBj}, 
there exists a constant $C>0$ such that
for every $f\in \cY'$, for every $G\in \cX'$,
for every $\bsy\in U$,
for every $s\in\bbN$ and for every $h>0$,
the variational problem \eqref{eq:uhs}
admits a unique solution $u_s^h(\bsy)$
which satisfies
\begin{equation}\label{eq:Idimtrunc}
  |I(G(u^h))- I(G(u^h_s))|
  \,\le\, C\, 
  \|f\|_{\cY'}\, \|G\|_{\cX'}\,
  \bigg(\sum_{j\ge s+1} \beta_{0,j}\bigg)^2
\end{equation}
for some constant $C>0$ 
independent of $f$, $G$ and of $s$
where $\beta_{0,j}$ is defined in
\eqref{eq:Bjsmall}. In addition, if \eqref{eq:ordered} and
\eqref{eq:psumpsi0} hold with $p_0<1$, then
\[
  \sum_{j\ge s+1} \beta_{0,j}
  \,\le\,
  \min\left(\frac{1}{1/p_0-1},1\right)
  \bigg(\sum_{j\ge1} \beta_{0,j}^{p_0} \bigg)^{1/p_0}
  s^{-(1/{p_0}-1)}\;.
\]
\end{theorem}

\subsection{Higher order QMC} \label{Sc:HiOrdQMC}
Higher order QMC rules were first studied in \cite{D08}.
Interlaced polynomial lattice rules are a special construction method of
higher order QMC rules which were first introduced in \cite{DiGo12} and further studied in \cite{Go13} and \cite{DKLNS13}.
The results in \cite{DKLNS13} use a non-Hilbert space setting
and bounds from \cite{D09}.
Following \cite{DKLNS13}, we consider numerical integration for smooth
integrands $F$ of $s$ variables defined over the unit cube $[0,1]^s$,
using a family of higher order digital nets called \emph{interlaced
polynomial lattice rules}. Below we only summarize the error bound, and
will not give any detail about interlaced polynomial lattice rules; the
full details can be found in \cite{DKLNS13}, for more background
information see also \cite{DiPi10}.

In particular, we are interested in integrands of the form $F(\bsy) =
G(u^h_s(\bsy-\bhalf))$. A novel non-Hilbert space setting was developed in
\cite{DKLNS13} to cater for such integrands. Let $\alpha, s\in\bbN$, and
$1\le q,r \le \infty$, and let $\bsgamma =
(\gamma_\setu)_{\setu\subset\bbN}$ be a collection of non-negative real
numbers, known as \emph{weights} (we refer to \cite{SW98} where the
concept was first introduced, and e.g., to \cite{DKS13} for
generalizations). Assume further that $F: [0,1]^s \to \mathbb{R}$ has
partial derivatives of orders up to $\alpha$ with respect to each
variable. Following \cite{DKLNS13}, we quantify the derivatives with
the norm of $F$ given by\footnote{The norm in
\cite[Definition~3.3]{DKLNS13} was incorrectly stated. The correct norm is
as given in \eqref{eq:defFabs} above. Since the correct norm was used in
the proof of \cite[Theorem~3.5]{DKLNS13}, all results in \cite{DKLNS13}
remain unaffected.}
\begin{align}\label{eq:defFabs}
 &\|F\|_{s,\alpha,\bsgamma,q,r} \,:=\,
 \Bigg[ \sum_{\setu\subseteq\{1:s\}} \Bigg( \frac{1}{\gamma_\setu^q}
 \sum_{\setv\subseteq\setu} \sum_{\bstau_{\setu\setminus\setv} \in \{1:\alpha\}^{|\setu\setminus\setv|}}
 \nonumber \\
 &\qquad\qquad\qquad\quad
 \int_{[0,1]^{|\setv|}} \bigg|\int_{[0,1]^{s-|\setv|}}
 (\partial^{(\bsalpha_\setv,\bstau_{\setu\setminus\setv},\bszero)}_\bsy F)(\bsy) \,
 \rd \bsy_{\{1:s\} \setminus\setv}
 \bigg|^q \rd \bsy_\setv \Bigg)^{r/q} \Bigg]^{1/r},
\end{align}
with the obvious modifications if $q$ or $r$ is infinite. Here
$\{1:s\}$ is a shorthand notation for the set $\{1,2,\ldots,s\}$, and
$(\bsalpha_\setv,\bstau_{\setu\setminus\setv},\bszero)$ denotes a sequence
$\bsnu$ with $\nu_j = \alpha$ for $j\in\setv$, $\nu_j = \tau_j$ for
$j\in\setu\setminus\setv$, and $\nu_j = 0$ for $j\notin\setu$. Two forms
of weights were considered in~\cite{DKLNS13}: SPOD weights 
(first introduced in \cite{DKLNS13}) take the form
\[ 
 \gamma_\setu
 \,:=\,
 \sum_{\bsnu_\setu \in \{1:\alpha\}^{|\setu|}}
 \Gamma_{|\bsnu_\setu|}  \prod_{j\in\setu} \gamma_j(\nu_j)\;,
\] 
while product weights take the form $\gamma_\setu := \prod_{j\in\setu}
\gamma_j$. We restrict to the case $r=\infty$, and we
use an abbreviated notation for the norm, namely, 
$\|F\|_{\calW_s} := \|F\|_{s,\alpha,\bsgamma,q,\infty}$.
\begin{theorem}[cp. {\cite[Theorems 3.5 and 3.9]{DKLNS13}}]\label{thm:wce}
Let $\alpha,s\in\bbN$ with $\alpha>1$, $1\le q\le\infty$ in
\eqref{eq:defFabs}, and let $\bsgamma = (\gamma_\setu)_{\setu\subset\bbN}$
denote a collection of weights. Let $b$ be prime and let $m\in\bbN$ be
arbitrary. Then, an interlaced polynomial lattice rule of order $\alpha$
with $N=b^m$ points $\{\bsy_0,\ldots,\bsy_{n-1}\}\in [0,1]^s$ can be
constructed using a component-by-component \textnormal{(}CBC\textnormal{)}
algorithm, such that
\begin{align*}
  \left| \int_{[0,1]^s} F(\bsy) \,\mathrm{d} \bsy - \frac{1}{b^m} \sum_{n=0}^{b^m-1} F(\bsy_n)\right|
  &\,\le\, \left(\frac{2}{b^{m}-1} \sum_{\emptyset\ne\setu\subseteq{\{1:s\}}}
  \gamma_\setu^\lambda\, [\rho_{\alpha,b}(\lambda)]^{|\setu|}\right)^{1/\lambda}\,
  \|F\|_{\calW_s},
\end{align*}
for all $1/\alpha < \lambda \le 1$, where
\begin{equation}\label{eq:defrhoab}
  \rho_{\alpha,b}(\lambda) \,:=\,
  \left(C_{\alpha,b}\,b^{\alpha(\alpha-1)/2}\right)^\lambda
  \left(\left(1+\frac{b-1}{b^{\alpha\lambda}-b}\right)^\alpha-1\right)\;,
\end{equation}
with
\begin{align*} 
 &C_{\alpha,b}
 \,:=\,
 \max\left(\frac{2}{(2\sin\frac{\pi}{b})^{\alpha}},\max_{1\le z\le\alpha-1}\frac{1}{(2\sin\frac{\pi}{b})^z}\right)
 \nonumber\\
 &\qquad\qquad\qquad\times
 \left(1+\frac{1}{b}+\frac{1}{b(b+1)}\right)^{\alpha-2}
 \left(3 + \frac{2}{b} + \frac{2b+1}{b-1} \right)\;.
\end{align*}
If the weights $\bsgamma$ are SPOD weights, then the CBC algorithm has
cost $\calO(\alpha\,s\, N\log N + \alpha^2\,s^2 N)$ operations. If the
weights $\bsgamma$ are product weights, then the CBC algorithm has cost
$\calO(\alpha\,s\, N\log N)$ operations.
\end{theorem}
\section{Error analysis}
\label{sec:MultAlg}
In this section, we analyse the error of the algorithm \eqref{eq:QL*}. 
For a geometric sequence
\[
  h_\ell = 2^{-\ell}\, h_0
  \qquad\mbox{for}\quad \ell = 1,2,\ldots
\]
of discretization parameters (such as, for example, the meshwidths of a
family of nested simplicial triangulations of the domain $D\subset
\mathbb{R}^d$), we assume given nested sequences
$\{\cX^{h_\ell}\}_{\ell\ge 0}\subset \cX$
and
$\{\cY^{h_\ell}\}_{\ell\ge 0}\subset \cY$
of subspaces of equal, increasing dimensions,
\[
  M_{0} < M_{1} < \cdots < M_{\ell}
:= {\rm dim}(\cX^{h_\ell}) =  {\rm dim}(\cY^{h_\ell}) \asymp 2^{d\ell}
\qquad \mbox{as}\quad \ell \to \infty
\;.
\]
This scaling of $M_{\ell}$ with respect to $\ell$ is typical for Galerkin
discretizations which are based on subspace sequences
obtained by (isotropic) mesh refinements in spatial dimension $d$.
We assume moreover
that the sequence $\{s_\ell\}_{\ell\ge 0}$ is nondecreasing,
\begin{equation} \label{eq:AssMonbeta}
  s_0\le s_1\le\cdots\le s_\ell \cdots.
\end{equation}
Since we are working with interlaced polynomial lattice rules, we assume
also that
\[
  N_\ell = b^{m_\ell}
  \qquad\mbox{for}\quad \ell = 0,1,2,\ldots.
\]

For the error analysis of algorithm $Q_*^L(G(u))$ defined in
\eqref{eq:QL*}, we rewrite using linearity of $I$, $G$ and of
$Q_{s_\ell,N_\ell}$
\begin{align} \label{eq:error-bound}
  &I(G(u)) - Q_*^L(G(u))
  \nonumber\\
  &\,=\, I(G(u-u^{h_L})) + I(G(u^{h_L}-u^{h_L}_{s_L}))
  +
  \sum_{\ell=0}^L (I - Q_{s_\ell,N_\ell}) (G(u_{s_\ell}^{h_\ell}-u_{s_{\ell-1}}^{h_{\ell-1}}))
  \;,
\end{align}
recalling that $u_{s_{-1}}^{h_{-1}} := 0$. For the first term in
\eqref{eq:error-bound} we estimate the integrand by the supremum over
$\bsy\in U$ and then apply \eqref{eq:Gconvest}. For the second term in
\eqref{eq:error-bound} we use \eqref{eq:Idimtrunc}. For each term in the
sum over $\ell$ in \eqref{eq:error-bound} we apply Theorem~\ref{thm:wce},
noting that here $I$ is effectively an $s_\ell$-dimensional integral since
the integrand depends only on the first $s_\ell$ variables. With
$\rho_{\alpha,b}(\lambda)$ as in \eqref{eq:defrhoab}, we then obtain the
bound
\begin{align} \label{eq:error-bound2}
  &|I(G(u)) - Q_*^L(G(u))|
  \nonumber\\
  &\,\leq \, C\,h_L^{\tau}\, \|f\|_{\cY'_t}\,\|G\|_{\cX'_{t'}}
  \,+\, C\, \|f\|_{\cY'}\,\|G\|_{\cX'}
  \bigg(\sum_{j\ge s_L+1}  \beta_{0,j} \bigg)^2
  \nonumber\\
  &\qquad + \sum_{\ell=0}^L
  \left(\frac{2}{N_\ell-1} \sum_{\emptyset\ne\setu\subseteq{\{1:s_\ell\}}}
  \gamma_\setu^\lambda\, [\rho_{\alpha,b}(\lambda)]^{|\setu|}\right)^{1/\lambda}\,
  \|G(u^{h_\ell}_{s_\ell} - u^{h_{\ell-1}}_{s_{\ell-1}})\|_{\calW_{s_\ell}}
  \;.
\end{align}

To estimate the final sum in the error estimate
\eqref{eq:error-bound2}, we bound for $\ell \ne 0$ the term
$\|G(u^{h_\ell}_{s_\ell} - u^{h_{\ell-1}}_{s_{\ell-1}})\|_{\calW_{s_\ell}}$.
The triangle inequality yields
\begin{equation} \label{eq:key}
  \|G(u^{h_\ell}_{s_\ell} - u^{h_{\ell-1}}_{s_{\ell-1}})\|_{\calW_{s_\ell}}
  \,\le\,
  \|G(u^{h_\ell}_{s_\ell} - u^{h_{\ell-1}}_{s_\ell})\|_{\calW_{s_\ell}}
  + \|G(u^{h_{\ell-1}}_{s_\ell} - u^{h_{\ell-1}}_{s_{\ell-1}})\|_{\calW_{s_\ell}}
\;,
\end{equation}
where the first term on the right-hand side of \eqref{eq:key} can again be
bounded by
\begin{align}\label{eq:key2}
  \|G(u^{h_\ell}_{s_\ell}-u^{h_{\ell-1}}_{s_\ell})\|_{\calW_{s_\ell}}
  \,\le\, \| G(u_{s_\ell} - u^{h_\ell}_{s_\ell})
    \|_{\calW_{s_\ell}} + \| G(u_{s_\ell} -
    u^{h_{\ell-1}}_{s_\ell}) \|_{\calW_{s_\ell}}
\;.
\end{align}
We estimate these terms in the next subsection.

\subsection{Two key theorems}
\label{sec:2KeyThms}


Theorems~\ref{thm:main1} and~\ref{thm:main2} below generalize
\cite[Theorems~7 and~8]{KSS13}.
In their proofs we use the following
lemma, which generalizes
\cite[Lemma~1]{KSS13}.

Let $\indx \,:=\, \{ \bsnu \in \bbN_0^\bbN \; : \; |\bsnu| < \infty \}$
denote the (countable) set of all ``finitely supported'' multi-indices
(i.e., sequences of non-negative integers for which only finitely many
entries are non-zero). For $\bsnu\in\indx$, let $\supp(\bsnu) := \{ j\in
\bbN: \nu_j \ne 0 \}$ denote the ``support'' of $\bsnu$. For
$\bsm,\bsnu\in\indx$, we write $\bsm\le\bsnu$ if $m_j\le \nu_j$ for all
$j$, we define $\sbinom{\bsnu}{\bsm}:= \prod_{j\ge 1}
\sbinom{\nu_j}{m_j}$, and we let $\bsnu-\bsm$ denote a multi-index with
the elements $\nu_j-m_j$.
We denote by $\bse_k$ the sequence whose
$k$th component is $1$ and all other components are $0$.
\begin{lemma} \label{lem:recur}
Given non-negative real numbers $(\Upsilon_j)_{j\in\bbN}$, let
$(\bbA_\bsnu)_{\bsnu\in\indx}$ and $(\bbB_\bsnu)_{\bsnu\in\indx}$ be
non-negative real numbers satisfying the inequality
\[
  \bbA_\bsnu
  \,\le\, \sum_{j\in\supp(\bsnu)} \nu_j\,\Upsilon_j\, \bbA_{\bsnu-\bse_j} + \bbB_\bsnu
  \quad\mbox{for any $\bsnu\in\indx$ \textnormal{(}including $\bsnu=\bszero$\textnormal{)}}.
\]
Then for any $\bsnu \in \indx$
\[
  \bbA_\bsnu
  \,\le\, \sum_{\bsm\le\bsnu} \sbinom{\bsnu}{\bsm}\, |\bsm|!\,
  \bsUpsilon^\bsm\,\bbB_{\bsnu-\bsm}\;,
  \quad\mbox{with}\quad \bsUpsilon^\bsm := \prod_{j\ge 1} \Upsilon_j^{m_j}\;.
\]
\end{lemma}

\begin{proof}
We prove this result by induction. The case $\bsnu = \bszero$ holds
trivially. Suppose that the result holds for all $|\bsnu| < n$ with some
$n\ge 1$. Then for $|\bsnu| = n$, we can use the inequality and the
induction hypothesis to write
\begin{align*}
  \bbA_\bsnu
  &\,\le\, \sum_{j\in\supp(\bsnu)} \nu_j\,\Upsilon_j\,
  \sum_{\bsm\le\bsnu-\bse_j}  |\bsm|! \, \sbinom{\bsnu-\bse_j}{\bsm} \,   \bsUpsilon^\bsm \,\bbB_{\bsnu-\bse_j-\bsm} + \bbB_\bsnu\;.
\end{align*}
Substituting $\bsm' = \bsm + \bse_j$, we can write
\begin{align*}
  \bbA_\bsnu
  &\,\le\, \sum_{j\in\supp(\bsnu)} \nu_j\,
  \sum_{\satop{\bsm'\le\bsnu}{m'_j\ge 1}} (|\bsm'|-1)!\,
  \frac{\sbinom{\nu_j-1}{m'_j-1}}{\sbinom{\nu_j}{m_j'}} \sbinom{\bsnu}{\bsm'}\,
  \bsUpsilon^{\bsm'}\,  \bbB_{\bsnu-\bsm'} + \bbB_\bsnu \\
  &\,=\, \sum_{\bszero\ne\bsm'\le\bsnu} \sum_{\satop{j\in\supp(\bsnu)}{m'_j\ge 1}}
  m_j'\,(|\bsm'|-1)!\, \sbinom{\bsnu}{\bsm'}\, \bsUpsilon^{\bsm'}\,  \,\bbB_{\bsnu-\bsm'} + \bbB_\bsnu \\
  &\,=\, \sum_{\bszero\ne\bsm'\le\bsnu} |\bsm'|\,
  (|\bsm'|-1)!\, \sbinom{\bsnu}{\bsm'}\, \bsUpsilon^{\bsm'}\,  \,\bbB_{\bsnu-\bsm'} + \bbB_\bsnu\;,
\end{align*}
which equals the desired formula.
\end{proof}

\begin{theorem}\label{thm:main1}
Under Assumptions~\ref{ass:AssBj} and~\ref{ass:XtYt} and the
conditions of Theorem~\ref{thm:Galerkin}, 
there exists $C>0$ such that
for every $f\in \cY'_{t}$
for every $G\in \cX'_{t'}$ with $0\le t,t'\le \bar{t}$,
for every $s\in \N$, and for every $h>0$ that is
admissible in the Galerkin discretization \eqref{eq:parmOpEqh}, 
there holds
\begin{align*}
  &\|G(u_s - u^h_s)\|_{\calW_s}
\\
  &\,\le\, C\, h^{t+t'}\,\|f\|_{\cY'_t}\, \|G\|_{\cX'_{t'}}
  \sup_{\setu\subseteq\{1:s\}}
  \frac{1}{\gamma_\setu}
  \sum_{\bsnu_\setu \in \{1:\alpha\}^{|\setu|}}
  (|\bsnu_\setu|+3)!\,\prod_{j\in\setu} \left(2^{\delta(\nu_j,\alpha)}\beta_{t,t',j}^{\nu_j}\right)
  \;,
\end{align*}
where $\beta_{t,t',j} :=
\max(\beta_{t,j},\beta_{t',j},\|A_j\|_{\cL(\cX,\cY')}/\bar{\mu},
\|A_j^*\|_{\cL(\cY,\cX')}/\bar{\mu})$, and $\delta(\nu_j,\alpha)$ is $1$
if $\nu_j=\alpha$ and is $0$ otherwise.
\end{theorem}

\begin{proof}
Let $g \in \cX'_{t'}$ denote the representer of the functional $G\in
\cX'$. For arbitrary $\bsy \in U$, define $v_g(\bsy) \in \cY$ and
$v_g^h(\bsy) \in \cY_h$ by
\begin{align*}
 \fa(\bsy;w,v_g(\bsy)) &\,=\, G(w) =
 {_{\cX'}}\langle g , w \rangle_{\cX} && \forall w \in \cX,
 \\
 \fa(\bsy;w^h,v_g^h(\bsy)) &\,=\,
 {_{\cX'}}\langle g, w_h\rangle_{\cX} && \forall w^h \in \cX^h
\;.
\end{align*}
Taking $w=u_s(\bsy)-u_s^h(\bsy)$, we have
\begin{align*}
G(u_s(\bsy)-u_s^h(\bsy))
&=\fa(\bsy;u_s(\bsy) - u_s^h(\bsy),v_g(\bsy))
\\
&=\fa(\bsy;u_s(\bsy) - u_s^h(\bsy),
v_g(\bsy)-v_g^h(\bsy))\;,
\end{align*}
where we used Galerkin orthogonality
$\fa(\bsy;u_s(\bsy) - u_s^h(\bsy) , v_g^h ) = 0$.
Using the
definitions of the bilinear form and the norm, we have
\begin{align*}
  &\|G(u_s - u^h_s)\|_{\calW_s} \,=\,\sup_{\setu\subseteq\{1:s\}}
     \Bigg[
       \frac{1}{\gamma_\setu} \Bigg(
       \sum_{\setv\subseteq\setu} \sum_{\bstau_{\setu\setminus\setv} \in \{1:\alpha\}^{|\setu\setminus\setv|}} \\
  &\qquad\qquad\qquad\qquad
    \int_{[-\frac{1}{2},\frac{1}{2}]^{|\setv|}}
       \bigg|\int_{[-\frac{1}{2},\frac{1}{2}]^{s-|\setv|}}
          r_{{(\bsalpha_\setv,\bstau_{\setu\setminus\setv},\bszero)}}({\bsy_{\{1:s\}};\bszero}) \,\rd \bsy_{\{1:s\} \setminus\setv}
       \bigg|^q
       \rd\bsy_\setv
     \Bigg)^{1/q} \Bigg],
\end{align*}%
where we define, for any multi-index $\bsnu\in\indx$ and any $\bsy\in U$,
\begin{align} \label{ruv_1}
 r_\bsnu (\bsy)
 \,:=&\,
 \partial^{\bsnu}_\bsy
 {_{\cY'}}\langle A(\bsy)\,e^h(\bsy),
 e^h_g(\bsy)\rangle_{\cY} \nonumber\\
 \,=&\,
 \partial^{\bsnu}_{\bsy}
 \Ylangle A_0e^h(\bsy), e^h_g(\bsy) \Yrangle
 + \sum_{j\ge 1} \partial^{\bsnu}_{\bsy}
 \left(y_j \; \Ylangle A_j e^h(\bsy), e^h_g(\bsy) \Yrangle
 \right)
\;,
\end{align}
with the abbreviated notation $e^h(\bsy) \,:=\, (u-u^h)(\bsy)$ and
$e^h_g(\bsy) \,:=\, (v_g - v_g^h)(\bsy)$.
Applying the Leibniz product rule $\partial^\bsnu (PQ) = \sum_{\bsm
\le\bsnu} \sbinom{\bsnu}{\bsm} (\partial^{\bsnu - \bsm} P)
(\partial^{\bsm} Q)$, we obtain
\begin{align}
  &\mbox{Second term on the RHS of \eqref{ruv_1}}
  \,=\,
  \sum_{j\ge 1} \sum_{\bsm \le \bsnu} \sbinom{\bsnu}{\bsm}
  \left(\partial_{\bsy}^{\bsm} y_j\right)\,
  \partial_{\bsy}^{\bsnu-\bsm}
  \Ylangle A_j\,e^h(\bsy) , e^h_g(\bsy) \Yrangle
  \nonumber\\
  &\,=\,
  \sum_{j\ge 1}
  y_j\,
  \partial_{\bsy}^{\bsnu}
  \Ylangle A_j\,e^h(\bsy) , e^h_g(\bsy) \Yrangle
  +
  \sum_{j\in\supp(\bsnu)} \nu_j\,
  \partial_{\bsy}^{\bsnu-\bse_j}
  \Ylangle A_j\,e^h(\bsy) , e^h_g(\bsy) \Yrangle
  \;, \label{ruv_2}
\end{align}
where we noted that $\partial^{\bsm}_\bsy y_j$ is $y_j$ if
$\bsm = \bszero$, is $1$ if $\bsm=\bse_j$ and $\nu_j\ge 1$,
and equals $0$ otherwise.
Substituting \eqref{ruv_2} into \eqref{ruv_1} and applying again the
product rule gives
\begin{align*}
  r_{\bsnu}(\bsy)
  &\,=\,
  \sum_{\bsm \le\bsnu} \sbinom{\bsnu}{\bsm}\,
  \Ylangle A_0 \partial^{\bsm}_{\bsy} e^h(\bsy),
  \partial_{\bsy}^{\bsnu - \bsm} e^h_g(\bsy) \Yrangle \\
  &\qquad +
  \sum_{j\ge 1} y_j
  \sum_{\bsm \le\bsnu} \sbinom{\bsnu}{\bsm}\,
  \Ylangle A_j \partial_{\bsy}^{\bsm} e^h(\bsy), \partial_{\bsy}^{\bsnu-\bsm}e^h_g(\bsy) \Yrangle
  \\
 &\qquad +
  \sum_{j\in\supp(\bsnu)} \nu_j\,
  \sum_{\bsm \le\bsnu-\bse_j} \sbinom{\bsnu { - \bse_j } }{\bsm}\,
  \Ylangle A_j \partial_{\bsy}^{\bsm} e^h(\bsy), \partial_{\bsy}^{\bsnu {- \bse_j} -\bsm}e^h_g(\bsy) \Yrangle
\;.
\end{align*}
Combining the first two terms and then using the continuity of the
operators $\{A_j\}_{j\ge 0}$, we conclude that
\begin{align} \label{ruv_3}
  |r_{\bsnu}(\bsy)|
  &\,\le\, \|A(\bsy)\|_{\cL(\cX,\cY')}
  \sum_{\bsm \le\bsnu} \sbinom{\bsnu}{\bsm}\,
  \|\partial^{\bsm}_{\bsy}e^h(\bsy)\|_{\cX}\,
  \|\partial_{\bsy}^{\bsnu - \bsm}e^h_g(\bsy)\|_{\cY}
  \nonumber
  \\
  &\quad +
  \sum_{j\in\supp(\bsnu)} \nu_j\,\|A_j\|_{\cL(\cX,\cY')}
  \sum_{\bsm \le\bsnu-\bse_j} \sbinom{\bsnu { - \bse_j} }{\bsm}\,
  \|\partial_{\bsy}^{\bsm} e^h(\bsy)\|_{\cX}\,
  \|\partial_{\bsy}^{\bsnu {- \bse_j} -\bsm}e^h_g(\bsy)\|_{\cY}
 \;.
\end{align}

To continue, we bound
$\|\partial_{\bsy}^{\bsm} e^h(\bsy)\|_{\cX}
 = \|\partial_{\bsy}^{\bsm} (u-u^h)(\bsy)\|_{\cX}$.
Let $\calI:\cX\to\cX$ denote the identity operator, and
let $\calP^h = \calP^h(\bsy) :\cX\to\cX^h$ denote the parametric Galerkin
projection defined, for any $w\in \cX$ and for every $\bsy\in U$
by\footnote{Note carefully that the projection $\calP^h$ depends on $\bsy$;
in order to not overburden the notation, we shall not indicate this
dependence explicitly.}
\begin{equation} \label{eq:Ph}
\calP^h w \in \cX^h: \quad
\fa(\bsy; \calP^h w, z^h)
=
\fa(\bsy;w,z^h) \quad\forall\,z^h\in \cY^h
\;.
\end{equation}
Then we arrive at $u^h(\bsy) = \calP^h u(\bsy)\in \cX^h$ and
$\partial_{\bsy}^{\bsm} u^h\in \cX^h$, giving
$(\calI - \calP^h)\partial_{\bsy}^{\bsm} u^h = 0$.
Thus
\begin{align} \label{eq:new_1}
  \|\partial_{\bsy}^{\bsm} e^h(\bsy)\|_{\cX}
  & \,=\, \| \calP^h \partial_{\bsy}^{\bsm} e^h(\bsy)
  + (\calI - \calP^h) \partial_{\bsy}^{\bsm} u(\bsy)\|_{\cX} \nonumber
\\
  & \,\le\, \| \calP^h \partial_{\bsy}^{\bsm} e^h(\bsy)\|_{\cX}
  \,+\, \|(\calI - \calP^h) \partial_{\bsy}^{\bsm} u(\bsy)\|_{\cX}
\;.
\end{align}
Recall that Galerkin orthogonality gives
{$\Ylangle A(\bsy)\, e^h(\bsy), z^h\Yrangle = 0$
for all $z^h\in \cY^h$ and for all $\bsy\in U$}.
Taking the derivative $\partial_{\bsy}^{\bsm}$ and
following similar steps to \eqref{ruv_1}
and \eqref{ruv_2},
we obtain for all $z^h\in \cY^h$
and for all $\bsy\in U$ that
\begin{align} \label{eq:new_2}
  \Ylangle A(\bsy)\,\partial_{\bsy}^{\bsm}e^h(\bsy) , z^h\Yrangle
  &\,=\, - \sum_{j\in\supp(\bsm)} m_j\,
  \Ylangle A_j\,\partial_{\bsy}^{\bsm-\bse_j}e^h(\bsy) , z^h \Yrangle\;.
\end{align}
Using again the definition \eqref{eq:Ph} of $\calP^h$,
we may replace
$\partial_{\bsy}^{\bsm}e^h(\bsy)$ on the left-hand side of
\eqref{eq:new_2} by $\calP^h\partial_{\bsy}^{\bsm}e^h(\bsy)$.
From the discrete inf-sup condition in Theorem~\ref{thm:Galerkin} (which
holds uniformly with respect to $\bsy\in U$) with constant $\bar{\mu}>0$,
it follows that there are constants $c_1, c_2>0$, independent of $h$ and
$\bsy$ and satisfying $\bar{\mu} = c_2/c_1$, such that for every $\bsy\in
U$ and $h>0$ and given $\calP^h\partial_{\bsy}^{\bsm}e^h(\bsy)\in \cX^h$
there exists $z^h = \zeta^h(\bsy) \in \cY^h$ for which
$\|\zeta^h(\bsy)\|_\cY \leq c_1 \|
\calP^h\partial_{\bsy}^{\bsm}e^h(\bsy)\|_\cX$ and $\Ylangle
A(\bsy)\,\partial_{\bsy}^{\bsm}e^h(\bsy), \zeta^h(\bsy)\Yrangle \ge c_2\,
\|\calP^h\partial_{\bsy}^{\bsm}e^h(\bsy)\|_{\cX}^2$. These together with
\eqref{eq:new_2} give
\begin{align*}
  c_2\, \|\calP^h\partial_{\bsy}^{\bsm}e^h(\bsy)\|_{\cX}^2
  &\,\le\, c_1 \sum_{j\in\supp(\bsm)} m_j \,\|A_j\|_{\cL(\cX,\cY')}
  \|\partial_{\bsy}^{\bsm-\bse_j}e^h(\bsy)\|_{\cX}\,\,
  \|\calP^h\partial_{\bsy}^{\bsm}e^h(\bsy)\|_{\cX}\;,
\end{align*}
which in turn yields for every $\bsy\in U$ the bound
\begin{align} \label{eq:new_3}
  \|\calP^h\partial_{\bsy}^{\bsm}e^h(\bsy)\|_{\cX}
  &\,\le\, \sum_{j\in\supp(\bsm)} m_j\,\frac{\|A_j\|_{\cL(\cX,\cY')}}{{\bar{\mu}}}
  \|\partial_{\bsy}^{\bsm-\bse_j}e^h(\bsy)\|_{\cX}\;.
\end{align}
Substituting \eqref{eq:new_3} into \eqref{eq:new_1} and then applying
Lemma~\ref{lem:recur}, we obtain
\begin{align*}
  \|\partial_{\bsy}^{\bsm}e^h(\bsy)\|_{\cX}
  & \,\le\, \sum_{\bsm'\le\bsm}\, \sbinom{\bsm}{\bsm'}\, |\bsm'|!\,
  \prod_{j\ge 1} \bigg(\frac{\|A_j\|_{\cL(\cX,\cY')}}{\bar{\mu}}\bigg)^{m'_j}
  \|(\calI - \calP^h) \partial_{\bsy}^{\bsm-\bsm'} u(\bsy)\|_{\cX}\;.
\end{align*}
Now from \eqref{eq:reg1} and \eqref{ass:diffy} we have
\begin{align*}
  \|(\calI - \calP^h) \partial_{\bsy}^{\bsm-\bsm'} u(\bsy)\|_{\cX}
 \le \bar{C}_t\, h^t \|\partial_{\bsy}^{\bsm-\bsm'}u(\bsy)\|_{\cX_t}
 \le C_t\, h^t\, \|f\|_{\cY'_t}\, |\bsm-\bsm'|!\,\bsbeta_t^{\bsm-\bsm'}\!\!.
\end{align*}
Moreover, we have
\begin{align}\label{eq:combin}
  \sum_{\bsm'\le\bsm} \sbinom{\bsm}{\bsm'}\, |\bsm'|!\,|\bsm-\bsm'|!
  &\,=\, \sum_{i=0}^{|\bsm|} \sum_{\satop{\bsm'\le\bsm}{|\bsm'|=i}} \sbinom{\bsm}{\bsm'}\, i!\,(|\bsm|-i)! \nonumber\\
  &\,=\, \sum_{i=0}^{|\bsm|} \sbinom{|\bsm|}{i}\, i!\,(|\bsm|-i)!
  \,=\, (|\bsm| +1)!\;,
\end{align}
where the second equality above follows from the identity
$\sum_{\bsm'\le\bsm, |\bsm'|=i} \sbinom{\bsm}{\bsm'} =
\sbinom{|\bsm|}{i}$.
%
Defining $_{1}\beta_{t,j}:=\max(\beta_{t,j},
\|A_j\|_{\cL(\cX,\cY')}/\bar{\mu})$,
we conclude that
\begin{align} \label{eq:nice_1}
  \|\partial_{\bsy}^{\bsm} e^h(\bsy)\|_{\cX}
  \,\le\, C_t\, h^t\, \|f\|_{\cY'_t}\, (|\bsm|+1)!\,_{1}\bsbeta_t^{\bsm}
\;.
\end{align}

Similarly, with $f$ replaced by $g$, $u$ replaced by $v_g$, $u^h$ replaced
by $v_g^h$, $\cX$ replaced by $\cY$, $\cX^h$ replaced by $\cY^h$, and
$\bsm$ replaced by $\bsnu-\bsm$, as well as \eqref{eq:reg1} and
\eqref{ass:diffy} replaced by \eqref{eq:reg2} and \eqref{eq:refAd}, we
obtain, after introducing the sequence
$_{2}\beta_{t',j}:=\max(\beta_{t',j},
\|A_j^*\|_{\cL(\cY,\cX')}/\bar{\mu})$,
\begin{align} \label{eq:nice_2}
  \|\partial_{\bsy}^{\bsnu-\bsm} (v_g-v_g^h)(\bsy)\|_{\cY}
  \,\le\, C_{t'}\, h^{t'}\, \|g\|_{\cX'_{t'}}\, (|\bsnu-\bsm|+1)!\,_{2}\bsbeta_{t'}^{\bsnu-\bsm}\;.
\end{align}

Using \eqref{eq:nice_1} and \eqref{eq:nice_2} and the identity
$\sum_{\bsm\le\bsnu} \sbinom{\bsnu}{\bsm}\, (|\bsm|+1)!\,(|\bsnu-\bsm|+1)!
= (|\bsnu| +3)!/6$, which can be obtained in the same way as
\eqref{eq:combin}, we conclude from \eqref{ruv_3}
\begin{align*}
  |r_\bsnu(\bsy)|
  &\,\le\, C_{t,t'}\, h^{t+t'}\,\|f\|_{\cY'_t}\, \|g\|_{\cX'_{t'}}
  \bigg(
  \|A(\bsy)\|_{\cL(\cX,\cY')}
  \frac{(|\bsnu|+3)!}{6}\, \bsbeta_{t,t'}^{\bsnu} \nonumber\\
  &\qquad\qquad\qquad\qquad\qquad\qquad
   + \sum_{j\in\supp(\bsnu)} \nu_j\,\|A_j\|_{\cL(\cX,\cY')}
  \frac{(|\bsnu-\bse_j|+3)!}{6}\, \bsbeta_{t,t'}^{\bsnu-\bse_j}
  \bigg) \nonumber\\
  &\,\le\,
   \max\Big(\sup_{\bsz\in U} \| A(\bsz) \|_{\cL(\cX,\cY')},\bar\mu\Big)
   \, C_{t,t'}\, h^{t+t'}\,\|f\|_{\cY'_t}\, \|g\|_{\cX'_{t'}}
  (|\bsnu|+3)!\,\bsbeta_{t,t'}^{\bsnu}\;,
\end{align*}
where 
$\beta_{t,t',j} := \max(_{1}\beta_{t',j},\; _{2}\beta_{t',j}) =
\max(\beta_{t,j},\beta_{t',j},\|A_j\|_{\cL(\cX,\cY')}/\bar{\mu},
\|A_j^*\|_{\cL(\cY,\cX')}/\bar{\mu})$.
Since $A(\bsy)\in \cL(\cX,\cY')$ is uniformly bounded with respect to $\bsy\in U$, we
conclude that there exists a constant $C>0$ which is independent of $s$
and of $h$, such that
\begin{align*}
  \|G(u_s - u^h_s)\|_{\calW_s}
  &\,\le\, C\,h^{t+t'}\,\|f\|_{\cY'_t}\, \|g\|_{{\cX'_{t'}} }
  \\
  &\quad\times
  \sup_{\setu\subseteq\{1:s\}}
  \frac{1}{\gamma_\setu}
  \sum_{\setv\subseteq\setu} \sum_{\bstau_{\setu\setminus\setv} \in \{1:\alpha\}^{|\setu\setminus\setv|}}
  (|(\bsalpha_\setv,\bstau_{\setu\setminus\setv},\bszero)|+3)!\,
  \bsbeta_{t,t'}^{(\bsalpha_\setv,\bstau_{\setu\setminus\setv},\bszero)}\;,
\end{align*}
where the last double sum can be rewritten as
\[
 \sum_{\bsnu_\setu \in \{1:\alpha\}^{|\setu|}} 2^{|\{j\in\setu\,:\,\nu_j=\alpha\}|}
 (|\bsnu_\setu|+3)!\,\bsbeta_{t,t'}^{\bsnu_\setu}
 \,=\,
 \sum_{\bsnu_\setu \in \{1:\alpha\}^{|\setu|}}
 (|\bsnu_\setu|+3)!\,\prod_{j\in\setu} \left(2^{\delta(\nu_j,\alpha)}\beta_j^{\nu_j}\right)\;,
\]
where $\delta(\nu_j,\alpha)$ is $1$ if $\nu_j=\alpha$ and is $0$
otherwise. This completes the proof.
\end{proof}
\begin{theorem}\label{thm:main2}
Under Assumptions~\ref{ass:AssBj} and \ref{ass:XtYt} and the
conditions of Theorem~\ref{thm:Galerkin}, 
there exists a constant $C>0$ such that
for every $f\in \cY'$, every $G \in \cX'$, every $h>0$, and for every $\ell \ge 1$, 
\begin{align*}
 &\|G(u^h_{\sell} - u^h_{\msell})\|_{\calW_{s_\ell}}
 \,\le\, C\ \|f\|_{\cY'}\, \|G\|_{\cX'} \\
  &\qquad \times
  \max\Bigg(\bigg(\sum_{j=s_{\ell-1}+1}^{s_\ell} \beta_{0,j}\bigg)
  \sup_{\setu\subseteq\{1:s_{\ell-1}\}}
     \frac{1}{\gamma_\setu}
     \sum_{\bsnu_\setu \in \{1:\alpha\}^{|\setu|}}
     (|\bsnu_\setu|+1)!\,\prod_{j\in\setu} \big( 2^{\delta(\nu_j,\alpha)}\,\bar\beta_{0,j}^{\nu_j} \big),
  \\
  &\qquad\qquad\qquad\qquad\qquad\qquad\quad
  \sup_{\satop{\setu\subseteq\{1:s_{\ell}\}}{\setu\cap\{s_{\ell-1}+1:s_\ell\}\ne\emptyset}}
     \frac{1}{\gamma_\setu}
     \sum_{\bsnu_\setu \in \{1:\alpha\}^{|\setu|}}
     |\bsnu_\setu|!\,\prod_{j\in\setu} \big( 2^{\delta(\nu_j,\alpha)}\,\beta_{0,j}^{\nu_j} \big)
  \Bigg)
   \;,
\end{align*}
where
$\bar\beta_{0,j} := \max(\beta_{0,j},\|A_j\|_{\cL(\cX,\cY')}/\bar\mu)$,
and
$\delta(\nu_j,\alpha)$ equals $1$ if $\nu_j=\alpha$ and equals $0$ otherwise.
\end{theorem}

\begin{proof}
Recalling the definition of the truncated bilinear form \eqref{equ:defAs},
for any $\bsy\in U$, $u^h_{\sell}(\bsy)$ and
$u^h_{\msell}(\bsy)$ are the solutions of the variational problems:
\begin{align}
\fa_{\sell}(\bsy; u^h_{\sell}(\bsy), v^h ) &= \Ylangle f, v^h \Yrangle
\qquad \forall v^h \in \cY^h\;, \label{eq:one}
\\
\fa_{\msell}(\bsy; u^h_{\msell}(\bsy), v^h ) &= \Ylangle f, v^h \Yrangle
\qquad \forall v^h \in \cY^h\;. \label{eq:two}
\end{align}
To estimate $\|G(u^h_{\sell} - u^h_{\msell})\|_{\calW_{s_\ell}}$, we make
use of the inequality
\[
  |\partial^\bsnu_\bsy
  (G(u^h_{\sell} - u^h_{\msell})(\bsy))|
  \,\le\, \|G\|_{\cX'}\, \|\partial^\bsnu_\bsy
  (u^h_{\sell} - u^h_{\msell})(\bsy)\|_{\cX}
\;.
\]

If $\supp(\bsnu)\cap\{s_{\ell-1}+1:s_\ell\}\ne\emptyset$, then it follows
from an adaption of \eqref{eq:Dsibound} for the Petrov-Galerkin
discretization that
\begin{equation} \label{eq:piece1}
  \|\partial^\bsnu_\bsy (u^h_{\sell} - u^h_{\msell})(\bsy)\|_{\cX}
  \,=\, \|\partial^\bsnu_\bsy u^h_{\sell}(\bsy)\|_{\cX}
  \,\le\, C_0\,|\bsnu|!\,\bsbeta_0^{\bsnu}\,\|f\|_{\cY'}\;.
\end{equation}
On the other hand, if $\supp(\bsnu)\subseteq\{1:s_{\ell-1}\}$, then we
subtract \eqref{eq:two} from \eqref{eq:one}
to obtain for every $\bsy \in U$ the equation
$\Ylangle A^{(\sell)}(\bsy) u^h_{\sell}(\bsy)-
   A^{(\msell)}(\bsy) u^h_{\msell}(\bsy), v^h \Yrangle =0 $
for all $v^h \in \cY^h$, or equivalently,
\[
\Ylangle A^{(\sell)}(\bsy)((u^h_{\sell}- u^{h}_{\msell})(\bsy)), v^h \Yrangle
=
-\Ylangle (A^{(\sell)}(\bsy) - A^{(\msell)}(\bsy)) u^{h}_{\msell}(\bsy), v^h \Yrangle
\;.
\]
Upon differentiating with respect to $\partial^\bsnu_\bsy$ for $\bsnu$
with $\supp(\bsnu)\subseteq\{1:s_{\ell-1}\}$, we obtain
\begin{align*}
 &\Ylangle A^{(\sell)}(\bsy)(\partial^\bsnu_\bsy(u^h_{\sell}- u^{h}_{\msell})(\bsy)), v^h \Yrangle
\\
 &\,=\, - \sum_{j\in\supp(\bsnu)} \nu_j\,
       \Ylangle A_j(\partial^{\bsnu-\bse_j}_\bsy(u^h_{\sell}- u^{h}_{\msell})(\bsy)), v^h \Yrangle
\\
 &\quad\;\, -\Ylangle (A^{(\sell)}(\bsy) - A^{(\msell)}(\bsy))\partial^\bsnu_\bsy u^{h}_{\msell}(\bsy), v^h \Yrangle
\;.
\end{align*}
Using the discrete inf-sup condition with parameter $\bar\mu>0$ as
in the proof of Theorem~\ref{thm:main1}, we choose $v^h$ to yield
\begin{align*}
 &\bar\mu\, \| \partial^\bsnu_\bsy(u^h_{\sell}- u^{h}_{\msell})(\bsy)\|_{\cX}^2
\\
 &\,\le\, \sum_{j\in\supp(\bsnu)} \nu_j\,
       \| A_j(\partial^{\bsnu-\bse_j}_\bsy(u^h_{\sell}- u^{h}_{\msell})(\bsy))\|_{\cY'}\,
       \| \partial^\bsnu_\bsy(u^h_{\sell}- u^{h}_{\msell})(\bsy)\|_{\cX}
\\
 &\quad\quad +\| (A^{(\sell)}(\bsy) - A^{(\msell)}(\bsy))\partial^\bsnu_\bsy u^{h}_{\msell}(\bsy)\|_{\cY'}\,
       \| \partial^\bsnu_\bsy(u^h_{\sell}- u^{h}_{\msell})(\bsy)\|_{\cX}
\;.
\end{align*}
Cancelling one common factor and applying further estimations, we obtain
\begin{align*}
 \|\partial^\bsnu_\bsy(u^h_{\sell}- u^{h}_{\msell})(\bsy)\|_{\cX}
 &\,\le\, \sum_{j\in\supp(\bsnu)} \nu_j\,
       \frac{\| A_j\|_{\cL(\cX,\cY')}}{\bar\mu}\,
       \|\partial^{\bsnu-\bse_k}_\bsy(u^h_{\sell}- u^{h}_{\msell})(\bsy)\|_{\cX}
\\
 &\quad\quad + \frac{1}{2} \sum_{j=s_{\ell-1}+1}^{s_\ell} \frac{\| A_j\|_{\cL(\cX,\cY')}}{\bar\mu}\,
 \|\partial^\bsnu_\bsy u^{h}_{\msell}(\bsy)\|_{\cX}
\;.
\end{align*}
Defining $\bar\beta_{0,j} :=
\max(\beta_{0,j},\|A_j\|_{\cL(\cX,\cY')}/\bar\mu)$, applying
Lemma~\ref{lem:recur}, and using again an adaption of
\eqref{eq:Dsibound} and the identity \eqref{eq:combin},
we obtain
\begin{align} \label{eq:piece2}
 &\|\partial^\bsnu_\bsy(u^h_{\sell}- u^{h}_{\msell})(\bsy)\|_{\cX} \nonumber
\\
 &\,\le\, \sum_{\bsm\le\bsnu} \sbinom{\bsnu}{\bsm}\, |\bsm|!\,\bar\bsbeta_0^\bsm
 \bigg(\frac{\|A_0\|_{\cL(\cX,\cY')}}{2\,\bar\mu} \sum_{j=s_{\ell-1}+1}^{s_\ell} \beta_{0,j}\,
 C_0\,|\bsnu-\bsm|!\,\bsbeta_0^{\bsnu-\bsm}\,\|f\|_{\cY'}\bigg) \nonumber\\
 &\,\le\, \frac{\|A_0\|_{\cL(\cX,\cY')}\,C_0}{2\,\bar\mu}\,\|f\|_{\cY'} \,(|\bsnu|+1)!\,\bar\bsbeta_0^\bsnu
 \sum_{j=s_{\ell-1}+1}^{s_\ell} \beta_{0,j}\;.
\end{align}

Combining \eqref{eq:piece1} and \eqref{eq:piece2}, 
we conclude that
\[
  \|G(u^h_{\sell} - u^h_{\msell})\|_{\calW_{s_\ell}} \,\le\, C\,\|f\|_{\cY'}\,\|G\|_{\cX'} \max(S_1,S_2)\;,
\]
with
\begin{align*}
  S_1 &:= \sum_{j=s_{\ell-1}+1}^{s_\ell} \beta_{0,j}
  \sup_{\setu\subseteq\{1:s_{\ell-1}\}}
     \frac{1}{\gamma_\setu}
     \sum_{\setv\subseteq\setu} \sum_{\bstau_{\setu\setminus\setv} \in \{1:\alpha\}^{|\setu\setminus\setv|}}
     (|(\bsalpha_\setv,\bstau_{\setu\setminus\setv},\bszero)|+1)!\,\bar\bsbeta_0^{(\bsalpha_\setv,\bstau_{\setu\setminus\setv},\bszero)}, \\
  S_2 &:= \sup_{\satop{\setu\subseteq\{1:s_{\ell}\}}{\setu\cap\{s_{\ell-1}+1:s_\ell\}\ne\emptyset}}
     \frac{1}{\gamma_\setu}
     \sum_{\setv\subseteq\setu} \sum_{\bstau_{\setu\setminus\setv} \in \{1:\alpha\}^{|\setu\setminus\setv|}}
     |(\bsalpha_\setv,\bstau_{\setu\setminus\setv},\bszero)|!\,\bsbeta_0^{(\bsalpha_\setv,\bstau_{\setu\setminus\setv},\bszero)}\;,
\end{align*}
which can be simplified to yield the desired result.
\end{proof}

\subsection{Error analysis of multi-level algorithm $Q^L_*$}
\label{seC:ErrMLAlgQL*}
In this section, we continue the error analysis of algorithm $Q^L_*$
defined in \eqref{eq:QL*} from the error bounds
\eqref{eq:error-bound2}--\eqref{eq:key2}. For the $\ell\ge 1$ terms we
apply Theorems~\ref{thm:main1} and~\ref{thm:main2}.
For the $\ell = 0$ term in \eqref{eq:error-bound2},
we use
$$
|\partial^\bsnu_\bsy G(u^{h_0}_{s_0}(\bsy))|
  \le \|G\|_{\cX'}\, \|\partial^\bsnu_\bsy u^{h_0}_{s_0}(\bsy)\|_{\cX}\,
  \le
  C_0\,|\bsnu|!\,\bsbeta_0^\bsnu\,\|f\|_{\cY'}\,\|G\|_{\cX'}
$$
to obtain
\begin{align*}
  \|G(u^{h_0}_{s_0})\|_{\calW_{s_0}} 
  &\,\le\,
  C_0\,\|f\|_{\cY'}\,\|G\|_{\cX'}
  \sup_{\setu\subseteq\{1:s_0\}}\frac{1}{\gamma_\setu}
  \sum_{\bsnu_\setu\in\{1:\alpha\}^{|\setu|}}
  |\bsnu_\setu|!\,\prod_{j\in\setu} \big(2^{\delta(\nu_j,\alpha)}\beta_{0,j}^{\nu_j}\big)
\;.
\end{align*}
Combining all these estimates, together with $\|f\|_{\cY'}\lesssim
\|f\|_{\cY'_t}$ and $\|G\|_{\cX'}\lesssim \|G\|_{\cX'_{t'}}$, with the
constants implied in $\lesssim$ depending on $t$ and $t'$ but independent
of $f$ and of $G$, we obtain for all $\lambda\in (1/\alpha,1]$,
with $\rho_{\alpha,b}$ as in \eqref{eq:defrhoab} the error bound
\begin{align} \label{eq:error-combined}
  &|I(G(u)) - Q_*^L(G(u))| \\
  &\le C\,\|f\|_{\cY'_t}\,\|G\|_{\cX'_{t'}} \Biggr[
  h_L^{\tau} + \bigg(\sum_{j\ge s_L+1}  \beta_{0,j}\bigg)^2
  \nonumber\\
  & + \Bigg(\frac{1}{N_0}\sum_{\emptyset\ne\setu\subseteq\{1:s_0\}}
  \gamma_\setu^\lambda\, [\rho_{\alpha,b}(\lambda)]^{|\setu|}\Bigg)^{1/\lambda}\!\!
  \Bigg(\sup_{\setu\subseteq\{1:s_0\}}\frac{1}{\gamma_\setu}
  \sum_{\bsnu_\setu\in\{1:\alpha\}^{|\setu|}}\!\!
  |\bsnu_\setu|!\,\prod_{j\in\setu} \big(2^{\delta(\nu_j,\alpha)}\beta_{0,j}^{\nu_j}\big)\Bigg)
 \nonumber\\
  &+ \sum_{\ell=1}^L
  \Bigg(\frac{1}{N_\ell}\sum_{\emptyset\ne\setu\subseteq\{1:s_\ell\}}
  \gamma_\setu^\lambda\, [\rho_{\alpha,b}(\lambda)]^{|\setu|}\Bigg)^{1/\lambda}\,\nonumber\\
  &\quad \cdot \Biggr[
  h_{\ell-1}^{\tau}\,
  \Bigg(\sup_{\setu\subseteq\{1:s_\ell\}}\frac{1}{\gamma_\setu}
  \sum_{\bsnu_\setu\in\{1:\alpha\}^{|\setu|}}
  (|\bsnu_\setu|+3)!\,\prod_{j\in\setu} \big(2^{\delta(\nu_j,\alpha)}\beta_{t,t',j}^{\nu_j}\big)\Bigg)
  \nonumber\\
  &\quad\, + \max \Bigg(
  \Bigg(\sum_{j=s_{\ell-1}+1}^{s_\ell} \beta_{0,j}\Bigg)
  \sup_{\setu\subseteq\{1:s_\ell\}}\frac{1}{\gamma_\setu}
  \sum_{\bsnu_\setu\in\{1:\alpha\}^{|\setu|}}
  (|\bsnu_\setu|+1)!\,\prod_{j\in\setu}
  \big(2^{\delta(\nu_j,\alpha)}\bar\beta_{0,j}^{\nu_j}\big), \nonumber
  \\
  &\qquad\qquad\qquad\qquad\qquad\qquad
  \sup_{\satop{\setu\subseteq\{1:s_{\ell}\}}{\setu\cap\{s_{\ell-1}+1:s_\ell\}\ne\emptyset}}
    \frac{1}{\gamma_\setu}
    \sum_{\bsnu_\setu \in \{1:\alpha\}^{|\setu|}}
    |\bsnu_\setu|!\,\prod_{j\in\setu} \big( 2^{\delta(\nu_j,\alpha)}\,\beta_{0,j}^{\nu_j} \big)
  \Bigg) \Biggr] \Biggr]\,, \nonumber
\end{align}
where $\sum_{j=s_{\ell-1}+1}^{s_\ell} \beta_{0,j} := 0$ 
if $s_\ell = s_{\ell-1}$, and where
we adopt the convention that a supremum over the empty set equals $0$.

\begin{theorem}
Under Assumptions~\ref{ass:AssBj} and \ref{ass:XtYt} and the
conditions of Theorem~\ref{thm:Galerkin},
for $f\in \cY'_t$ and $G\in \cX'_{t'}$ with $0 \le t,t'\le \bar{t}$
and $\tau:=t+t'>0$, consider the
multi-level QMC Petrov-Galerkin algorithm defined by \eqref{eq:QL*},
with interlaced polynomial lattice rules as in Theorem~\ref{thm:wce} with
SPOD weights 
\begin{equation}\label{eq:choiceweight-1}
\gamma_{\setu} \,:=\,
\sum_{\bsnu_\setu\in\{1:\alpha\}^{|\setu|}} (|\bsnu_\setu|+3)!\,
\prod_{j\in\setu} \big(2^{\delta(\nu_j,\alpha)}\beta_j^{\nu_j}\big)
\;,
\end{equation}
where, for $j\geq 1$, the SPOD weight sequence $\bsbeta$ is given by
\begin{align} \label{eq:beta}
  \beta_j
  := \max\bigg(\beta_{0,j}^{p_0/q},\beta_{t,j},\beta_{t',j},\beta_{0,j},
  \frac{\|A_j\|_{\cL(\cX,\cY')}}{\bar\mu},
  \frac{\|A_j^*\|_{\cL(\cY,\cX')}}{\bar\mu}
  \bigg)\;,
\end{align}
for some parameter $q$ satisfying $p_t\le q\le 1$.
Then for all $\lambda$ satisfying $\lambda\ge q$ and $1/\alpha < \lambda\le 1$ we have
\begin{align} \label{eq:error-simp}
  &|I(G(u)) - Q_*^L(G(u))|
  \,\le\, C\,D_\bsgamma(\lambda)\,\|f\|_{\cY'_t}\,\|G\|_{\cX'_{t'}}\,  \nonumber
  \\
  &\qquad\qquad\cdot
  \left[
  \left(h_L^{\tau} + s_L^{-2(1/{p_0}-1)} \right)
  + \sum_{\ell=0}^L N_\ell^{-1/\lambda}
    \left(h_{\ell-1}^{\tau} + \theta_{\ell-1}\,s_{\ell-1}^{-(1/{p_0}-1/q)}\right)
    \right]\;,
\end{align}
where
\[ 
  D_\bsgamma(\lambda)
  \,:=\,
  \Bigg(\sum_{|\setu|<\infty}
  \gamma_\setu^\lambda\, [\rho_{\alpha,b}(\lambda)]^{|\setu|}\Bigg)^{1/\lambda}
  \,<\,\infty\;.
\] 
In general we have $\theta_\ell = 1$ for all $\ell=0,\ldots,L$, but if
$s_\ell = s_{\ell-1}$ for some $\ell\ge 1$ then $\theta_{\ell-1} = 0$.
Maximal convergence rates from these bounds can be obtained with the
choices
\begin{equation}\label{eq:lamalpha}
  q \,:=\, p_t\;, \quad
  \lambda \,:=\, p_t \quad\mbox{and}\quad \alpha \,:=\, \lfloor 1/p_t\rfloor + 1.
\end{equation}
\end{theorem}

\begin{proof}
First we observe that $\beta_j$ defined in \eqref{eq:beta} is greater
than or equal to $\beta_{0,j}$, $\beta_{t,t',j}$ of
Theorem~\ref{thm:main1}, and $\bar\beta_{0,j}$ of Theorem~\ref{thm:main2}.
Thus, with weights given by \eqref{eq:choiceweight-1}, all suprema in the
error bound \eqref{eq:error-combined} are bounded by $1$. The motivation
for introducing $\beta_{0,j}^{p_0/q}$ in \eqref{eq:beta} is to improve the
bound on the last supremum in \eqref{eq:error-combined}, noting that when
$q=p_t$, $\beta_{0,j}^{p_0/p_t}$ has the same decay property as
$\beta_{t,j}$.
We bound $S_2$ in the proof of Theorem~\ref{thm:main2} as follows:
\begin{align*}
  &S_2 \,=\, \sup_{\satop{\setu\subseteq\{1:s_{\ell}\}}{\setu\cap\{s_{\ell-1}+1:s_\ell\}\ne\emptyset}}
     \frac{1}{\gamma_\setu}
     \sum_{\bsnu_\setu \in \{1:\alpha\}^{|\setu|}}
     |\bsnu_\setu|!\,\prod_{j\in\setu} \big( 2^{\delta(\nu_j,\alpha)}\,\beta_{0,j}^{\nu_j} \big) \\
 &\,=\, \sup_{k\in \{s_{\ell-1}+1 : s_\ell\}}
     \sup_{k\in\setu\subseteq\{1:s_{\ell}\}}
     \frac{\sum_{\bsnu_\setu \in \{1:\alpha\}^{|\setu|}}
     |\bsnu_\setu|!\,\prod_{j\in\setu} \big( 2^{\delta(\nu_j,\alpha)}\,\beta_{0,j}^{\nu_j} \big)}
     {\sum_{\bsnu_\setu' \in \{1:\alpha\}^{|\setu|}}
     (|\bsnu_\setu'|+3)!\,\prod_{j\in\setu} \big(2^{\delta(\nu_j',\alpha)}\,\beta_{j}^{\nu_j'} \big)} \\
  &\,=\, \sup_{k\in \{s_{\ell-1}+1 : s_\ell\}} \\
  &\qquad
     \sup_{\setv\subseteq\{1:s_{\ell}\}\setminus\{k\}}
     \frac{
     \sum_{\nu_k=1}^\alpha 2^{\delta(\nu_k,\alpha)}\,\beta_{0,k}^{\nu_k}\,
     \sum_{\bsnu_\setv \in \{1:\alpha\}^{|\setv|}}
     (|\bsnu_\setv|+1)!\,\prod_{j\in\setv} \big( 2^{\delta(\nu_j,\alpha)}\,\beta_{0,j}^{\nu_j} \big)}
     {
     \sum_{\nu_k'=1}^\alpha 2^{\delta(\nu_k',\alpha)}\,\beta_k^{\nu_k'}\,
     \sum_{\bsnu_\setv' \in \{1:\alpha\}^{|\setv|}}
     (|\bsnu_\setv'|+4)!\,\prod_{j\in\setv} \big(2^{\delta(\nu_j',\alpha)}\,\beta_j^{\nu_j'} \big)} \\
  &\,\le\, \sup_{k\in \{s_{\ell-1}+1 : s_\ell\}}
     \frac{\sum_{\nu_k=1}^\alpha 2^{\delta(\nu_k,\alpha)}\,\beta_{0,k}^{\nu_k}}
          {\sum_{\nu_k'=1}^\alpha 2^{\delta(\nu_k',\alpha)}\,\beta_{k}^{\nu_k'}}
  \,\le\, \sup_{k\in \{s_{\ell-1}+1 : s_\ell\}} \sum_{\nu_k=1}^\alpha  \beta_{0,k}^{(1-p_0/q)\nu_k}\;,
\end{align*}
where we dropped the $\nu_k'\ne \nu_k$ terms in the denominator and used
$\beta_k\ge\beta_{0,k}^{p_0/q}$.
Using \eqref{eq:ordered} and assuming that $s_{\ell-1}$ is
sufficiently large so that $\beta_{0,s_{\ell-1}+1}<1$,
we obtain
\[
  S_2
  \,\le\, \alpha\, \beta_{0,s_{\ell-1}+1}^{1-p_0/q}
  \,=\, \alpha\, \beta_{0,s_{\ell-1}+1}^{p_0(1/p_0-1/q)}
  \,\le\, \alpha\, s_{\ell-1}^{-(1/p_0-1/q)} \bigg(\sum_{j\ge 1}\beta_{0,j}^{p_0}\bigg)^{1/p_0-1/q}\;.
\]
In comparison, the tail sum $\sum_{j=s_{\ell-1}+1}^{s_\ell} \beta_{0,j} =
\calO(s_{\ell-1}^{-(1/p_0-1)})$ has a better exponent, and therefore is
dominated by $S_2$. 
This yields the simplified error bound \eqref{eq:error-simp}.

We now show that $D_\bsgamma(\lambda) < \infty$ for $\lambda\ge p_t$ and
$1/\alpha < \lambda\le 1$.
Using Jensen's inequality we have
\begin{align*}
  [ D_\bsgamma(\lambda) ]^\lambda
  &\,=\, \sum_{|\setu|<\infty}  [\rho_{\alpha,b}(\lambda)]^{|\setu|}
  \bigg( \sum_{\bsnu_\setu \in \{1:\alpha\}^{|\setu|}}  {(|\bsnu_\setu|+3)!}
  \,\prod_{j\in\setu}  \big( 2^{\delta(\nu_j,\alpha)} {\beta_j^{\nu_j}}\big) \bigg)^\lambda \\
 &\,\le\, \sum_{|\setu|<\infty}   \sum_{\bsnu_\setu \in \{1:\alpha\}^{|\setu|}}
  [{(|\bsnu_\setu|+3)!}]^\lambda \,\prod_{j\in\setu} {\widetilde\beta_j^{\lambda\nu_j}}\;.
\end{align*}
where we introduced
$\widetilde\beta_j :=
\rho^{1/\lambda}_{\alpha,b}(\lambda)  2^{ \delta(\nu_j,\alpha)} \beta_j$
to simplify the notation.
We now define a sequence $d_j :=
\widetilde\beta_{\lceil j/\alpha\rceil}$
so that $d_1 = \cdots = d_\alpha = \widetilde\beta_1$
and
$d_{\alpha+1} = \cdots = d_{2\alpha} =
\widetilde\beta_2$, and so on.
Then any term of the form
$[(|\bsnu_\setu|+3)!]^\lambda \,\prod_{j\in\setu}
\widetilde{\beta}_j^{\lambda \nu_j}$ can be written as
$
  [(|\setv|+3)!]^\lambda\, \prod_{j\in\setv} d_j^\lambda
$ 
for some finite subset of indices $\setv\subset\bbN$.
Thus we conclude that
\begin{align} \label{eq:last}
  [D_\bsgamma(\lambda)]^\lambda
  &\,<\, \sum_{\satop{\setv\subset\bbN}{|\setv|<\infty}}
  \bigg((|\setv|+3)! \prod_{j\in\setv} d_j\bigg)^\lambda \nonumber
 \\
  &\,=\, \sum_{\ell=0}^\infty [(\ell+3)!]^\lambda
  \sum_{\satop{\setv\subset\bbN}{|\setv|=\ell}} \prod_{j\in\setv} d_j^\lambda
  \,\le\, \sum_{\ell=0}^\infty \frac{[(\ell+3)!]^\lambda}{\ell!}
  \bigg(\sum_{j=1}^\infty d_j^\lambda\bigg)^\ell\;.
\end{align}
Note that $\sum_{j=1}^\infty d_j^{\lambda} < \infty$ holds if and only
if $\sum_{j=1}^\infty \beta_j^{\lambda} < \infty$.
By the ratio test, the
last expression in \eqref{eq:last} is finite if $p_t \le q \le \lambda < 1$.
Alternatively, using the geometric series formula, the last
expression in \eqref{eq:last} is finite if $\lambda=1$ and
$\sum_{j=1}^\infty d_j < 1$. Recall that $\lambda$ also needs to satisfy
$1/\alpha < \lambda\le 1$.
This leads to the choice \eqref{eq:lamalpha}.
\end{proof}

\subsection{Optimizing the cost versus error bound}

Recall that
\begin{equation} \label{eq:def_h}
  h_\ell \,\asymp\, 2^{-\ell}
  \quad\mbox{and}\quad
  M_{h_\ell} \,\asymp\, h_\ell^{-d} \,\asymp\, 2^{\ell d}
  \quad\mbox{for}\quad \ell=0,\ldots,L\;.
\end{equation}
Based on the error bound \eqref{eq:error-simp}
with \eqref{eq:lamalpha}, we now specify
$s_\ell$ and $N_\ell$ for each level.

To balance the error contribution within the highest discretization level,
we impose the condition $s_L^{-2(1/p_0-1)} = \calO( h_L^{\tau} )$, which
is equivalent to $s_L = \Omega( 2^{L\tau p_0/(2-2p_0)})$. Then, to
minimize the error within each level, one choice for $s_\ell$ is to set
$s_\ell = s_L$ for all $\ell<L$, leading to $\theta_{\ell-1} = 0$ for all
$\ell=1,\ldots,L$ in \eqref{eq:error-simp}.

Alternatively, since $s_\ell$ should be as small as possible from the
point of view of reducing the cost at each level, we may
impose the condition
${s_{\ell-1}^{-(1/p_0-1/p_t)} = s_{\ell-1}^{-t / d }}
=
\calO(h_{\ell-1}^{\tau} )$ for $\ell=1,\ldots,L$,
which is equivalent to
$s_\ell = \Omega( 2^{\ell\tau d /t})$ for $\ell=0,\ldots,L-1$,
where we substituted $p_t = p_0 /(1-t p_0 / d)$, see \eqref{eq:pt}.

Combining both approaches, while taking into account the monotonicity
condition \eqref{eq:AssMonbeta}, we choose
\begin{align} \label{eq:def_s}
  s_\ell \,:=\,
  \min \Big( \big\lceil 2^{\ell\tau d /t}\big\rceil ,
  \big\lceil 2^{L\tau p_0/(2-2p_0)}\big\rceil  \Big)
  \quad\mbox{for}\quad \ell=0,\ldots,L\;.
\end{align}
Thus we have $s_\ell$ strictly increasing for
$\ell=0,\ldots,\min(\lfloor Ltp_0/( d (2-2p_0) ) \rfloor,L)$,
and the remaining $s_\ell$ (if any) are all identical.
Our choice of $s_\ell$ leads to the error bound
\[ 
  {\rm error} \,=\,
  \calO \left(
  h_L^{\tau}
  + \sum_{\ell=0}^L N_\ell^{-1/p_t}\, h_{\ell}^{\tau} \right)\;,
\] 
where we used $h_{\ell-1} \asymp h_{\ell}$.
For our cost model we assume the availability of a linear complexity
Petrov-Galerkin solver so that
\[
  {\rm cost} \,=\, \calO\left(\sum_{\ell=0}^L N_\ell\, h_\ell^{-d}\,s_\ell\right)\;.
\]

To \emph{minimize the error bound for a fixed cost}, we treat the
cost constraint by a Lagrange multiplier ${\theta}$ and consider the
function
\[
  g({\theta}) \,:=\, \underbrace{h_L^{\tau}
  + \sum_{\ell=0}^L N_\ell^{-1/p_t} h_{\ell}^{\tau}}_{\mbox{\footnotesize{error bound}}}
  \;+\; {\theta}\;
  \underbrace{\sum_{\ell=0}^L N_\ell\, h_\ell^{-d}\,s_\ell}_{\mbox{\footnotesize{cost}}}
  \;.
\]
We look for the stationary point of $g(\theta)$ with respect to $N_\ell$,
thus demanding that
\[
  \frac{\partial g(\theta)}{\partial N_\ell}
  \,=\,
  -\frac{1}{p_t} N_\ell^{-1/p_t-1} h_{\ell}^{\tau} +
  \theta\,h_\ell^{-d}\,s_\ell \,=\, 0
  \qquad\mbox{for}\quad \ell=0,\ldots,L\;.
 \]
This prompts us to define
\begin{equation} \label{eq:def_N}
  N_\ell
  \,:=\,
  \Big\lceil N_0
  \left(h_0^{-\tau-d}\,s_0\, h_\ell^{\tau+d}\,s_\ell^{-1} \right)^{p_t/(p_t+1)}
  \Big\rceil
  \qquad\mbox{for}\quad \ell=1,\ldots,L\;.
\end{equation}
Leaving $N_0$ to be specified later and treating $h_0$ and $s_0$ as
constants, we conclude that
\[ 
  {\rm error} \,=\, \calO \left(
  h_L^{\tau}
  \;+\; N_0^{-1/p_t}
  \sum_{\ell=0}^L E_\ell \right)
  \quad\mbox{and}\quad
  {\rm cost} \,=\, \calO \left( N_0\, \sum_{\ell=0}^L E_\ell \right)\;,
\]
where $ E_\ell := (h_\ell^{p_t \tau-d}\,s_\ell)^{1/(p_t+1)}$. The error is
\emph{not} necessarily minimized by balancing the error terms between the
levels.

We consider separately the two alternative choices in \eqref{eq:def_s}:
choice ${{\mathcal A}}$ takes $s_\ell = \lceil 2^{\ell\tau {d} /t} \rceil$ for
all $\ell$, while choice ${{\mathcal B}}$ takes $s_\ell = \lceil 2^{L\tau
\kappa} \rceil$ for all $\ell$, where
\[
  \kappa \,:=\, p_0/(2-2p_0)\;.
\]
Since $E_\ell$ increases with increasing $s_\ell$, we have
\begin{align*}
  \sum_{\ell=0}^L E_\ell &\,\le\, \min \left(\sum_{\ell=0}^L
  E_\ell^{({\mathcal A})},
  \sum_{\ell=0}^L E_\ell^{({\mathcal B})}\right)\;,
\end{align*}
where
\begin{align}
  \sum_{\ell=0}^L E_\ell^{({\mathcal A})}
  &\,=\, \calO \Bigg( \sum_{\ell=0}^L 2^{\ell\tau (d/\tau - p_t + d /t)/(p_t+1)} \Bigg) 
\nonumber\\
  &\,=\,
  \begin{cases}
  \calO \big(1 \big)
    & \mbox{if } d/\tau < p_t - d /t\;,  
\\
  \calO \big(L\big)
    & \mbox{if } d/\tau = p_t - d /t\;, 
\\
  \calO \big(2^{L\tau (d/\tau - p_t + d /t)/(p_t+1)} \big)
    & \mbox{if } d/\tau > p_t - d /t\;,
  \end{cases}
  \label{eq:EA}
  \\
  \sum_{\ell=0}^L E_\ell^{({\mathcal B})}
  &\,=\, \calO
  \Bigg( 2^{L\tau\kappa/(p_t+1)} \sum_{\ell=0}^L 2^{\ell\tau (d/\tau-p_t)/(p_t+1)} \Bigg)
  \nonumber\\
  &\,=\,
  \begin{cases}
  \calO \big(2^{L\tau\kappa/(p_t+1)}\big)
    & \mbox{if } d/\tau < p_t \;,  \\
  \calO \big(2^{L\tau\kappa/(p_t+1)} L \big)
    & \mbox{if } d/\tau = p_t\;,  \\
  \calO \big(2^{L\tau(d/\tau-p_t+\kappa)/(p_t+1)} \big)
    & \mbox{if } d/\tau > p_t\;.
  \end{cases}
  \label{eq:EB}
\end{align}
Thus we can take the minimum between \eqref{eq:EA} and \eqref{eq:EB} as
appropriate.

For the ``intermediate case'' $p_t - d /t < d/\tau < p_t$, if the
``crossover'' index in \eqref{eq:def_s}, i.e., $\ell=\min(\lfloor
L\kappa t / d \rfloor,L)$, is strictly less than $L$ (which happens when
$\kappa t < d$), it may be beneficial to take the alternative approach to
estimate directly
\begin{align*}
  \sum_{\ell=0}^L E_\ell
  &\,=\, \calO \Bigg( \sum_{\ell=0}^{\lfloor L\kappa t / d \rfloor} 
2^{\ell\tau (d/\tau - p_t + d /t)/(p_t+1)}
  + 2^{L\tau\kappa/(p_t+1)} \sum_{\ell=\lfloor L\kappa t/d \rfloor+1}^L 2^{\ell\tau (d/\tau-p_t)/(p_t+1)} 
     \Bigg)\\
  &\,=\,
  \calO\big(
  2^{L\tau\kappa t (1/\tau - p_t/ d + 1/t)/(p_t+1)}
  + 2^{L\tau\kappa/(p_t+1) + L\tau\kappa t (1/\tau-p_t/ d)/(p_t+1)}
  \big) \\
  &\,=\,
  \calO\big(2^{L\tau\kappa t (1/\tau - p_t/ d + 1/t)/(p_t+1)}\big)\;,
\end{align*}
which is always smaller than the first case of \eqref{eq:EB}, 
and is smaller than or equal to the third case of \eqref{eq:EA} when 
$\kappa t\le d$.
Hence we conclude that
\begin{align*}
  \sum_{\ell=0}^L E_\ell
  &\,=\,
  \begin{cases}
  \calO \big(1 \big)
    & \mbox{if } d/\tau < p_t - d  /t\;,  \\
  \calO \big(L\big)
    & \mbox{if } d/\tau = p_t -  d /t\;,  \\
  \calO \big(2^{L\tau t\min({d} /t,\kappa) (1/\tau - p_t / d + 1/t) /(p_t+1)} \big)
    & \mbox{if } p_t - d /t < d/\tau <  p_t \;,  \\
  \calO \big(2^{L\tau\min( d /t,\kappa)/(p_t+1)} L \big)
    & \mbox{if } d/\tau = p_t\;,  \\
  \calO \big(2^{L\tau[d/\tau-p_t+\min( d /t,\kappa)]/(p_t+1)} \big)
    & \mbox{if } d/\tau > p_t\;.
  \end{cases}
\end{align*}

We choose $N_0$ to satisfy
\[ 
  N_0^{-1/p_t}\, \sum_{\ell=0}^L E_\ell \,=\, \calO( h_L^{\tau} )\;,
\] 
which is equivalent to $ N_0 = \Omega ( h_L^{-\tau p_t} (\sum_{\ell=0}^L
E_\ell)^{p_t})$. 
This yields
\begin{align} \label{eq:def_N0}
  N_0
  &\,:=\,
  \begin{cases}
  \big\lceil 2^{L\tau p_t} \big\rceil
    & \mbox{if } d/\tau < p_t - d/t\;,  \\
  \big\lceil 2^{L\tau p_t} L^{p_t} \big\rceil
    & \mbox{if } d/\tau = p_t - {d}/t\;,  \\
  \big\lceil 2^{L\tau [p_t+1+t\min(d/t,\kappa)(1/\tau - p_t/d + 1/t)]p_t/(p_t+1)} \big\rceil
    & \mbox{if } p_t - d/t < d/\tau <  p_t\;,  \\
  \big\lceil 2^{L\tau [p_t+1+\min(d/t,\kappa)]p_t/(p_t+1)} L^{p_t}  \big\rceil
    & \mbox{if } d/\tau = p_t\;,  \\
  \big\lceil 2^{L\tau [1+d/\tau +\min(d/t,\kappa)]p_t/(p_t+1)} \big\rceil
    & \mbox{if } d/\tau > p_t\;.
  \end{cases}
\end{align}
Then we have ${\rm error} = \calO ( h_L^{\tau})$, and
\begin{align*}
  {\rm cost} &\,=\, \calO \big(
  N_0^{(p_t+1)/p_t} h_L^{\tau}\big) \\
  &\,=\,
  \begin{cases}
  \calO\big( 2^{L\tau p_t} \big)
    & \mbox{if } d/\tau < p_t - d/t\;,  \\
  \calO\big( 2^{L\tau p_t} L^{p_t+1} \big)
    & \mbox{if } d/\tau = p_t - d/t\;,  \\
  \calO\big( 2^{L\tau [p_t + t\min( d /t,\kappa)( 1/\tau - p_t/d + 1/t )]} \big)
    & \mbox{if } p_t - d/t < d/\tau <  p_t \;,  \\
  \calO\big( 2^{L\tau [p_t+\min( d/t,\kappa)]} L^{p_t+1} \big)
    & \mbox{if } d/\tau = p_t\;,  \\
  \calO\big( 2^{L\tau [d/\tau+\min(d/t,\kappa)]} \big)
    & \mbox{if } d/\tau > p_t\;.
  \end{cases}
\end{align*}
For given $\varepsilon > 0$, we choose $L$ such that
\begin{equation} \label{eq:def_L}
  h_L^\tau \,\asymp\, 2^{-L\tau} \,\asymp\, \varepsilon\;.
\end{equation}
We can then express the total cost of the algorithm in terms of
$\varepsilon$.
\begin{theorem} \label{thm:summary}
Under Assumptions~\ref{ass:AssBj} and \ref{ass:XtYt} and the
conditions of Theorem~\ref{thm:Galerkin},
for $f\in \cY'_t$ and $G\in \cX'_{t'}$ with
$0 \le t,t'\le \bar{t}$ and $\tau:=t+t'>0$, we consider
the multi-level QMC Petrov-Galerkin algorithm defined by \eqref{eq:QL*}. 

Given $\varepsilon>0$, with $L$ given by \eqref{eq:def_L},
$h_\ell$ given by \eqref{eq:def_h}, $s_\ell$ given by \eqref{eq:def_s},
$N_\ell$ given by \eqref{eq:def_N}, $N_0$ given by \eqref{eq:def_N0}, and
with interlaced polynomial lattice rules constructed based on SPOD weights
$\gamma_\setu$ given by \eqref{eq:choiceweight-1} with $q = p_t$, 
we obtain
\[
  |I(G(u)) - Q_*^L(G(u))|
  \,=\, \calO \left(\varepsilon \right)
  \;,
\]
and
\[
  {\rm cost}(Q_*^L)
  \,=\, \calO \big( \varepsilon^{-a^{\rm ML}}\, (\log\varepsilon^{-1})^{b^{\rm ML}} \big)\;,
\]
with the constants implies in $\calO(\cdot)$ being independent of $h_\ell$, $s_\ell$ and $N_\ell$,
and
\begin{align*}
  a^{\rm ML}
  &\,=\,
  \begin{cases}
  p_t
    & \mbox{if}\quad \displaystyle \frac{d}{\tau} \le p_t - \frac{d}{t}\;,  \vspace{0.1cm} \\
  \displaystyle p_t + t\min\Big(\frac{ d }{t},\frac{p_0}{2-2p_0}\Big)\Big( \frac{1}{\tau} - \frac{p_t}{d} + \frac{1}{t}\Big)
    & \mbox{if}\quad \displaystyle p_t - \frac{d}{t} < \frac{d}{\tau} <  p_t\;,  \\
  \displaystyle \frac{d}{\tau} + \min\Big(\frac{d}{t},\frac{p_0}{2-2p_0}\Big)
    & \mbox{if}\quad \displaystyle \frac{d}{\tau} \ge p_t\;.
  \end{cases}
\end{align*}
The value of $b^{\rm ML}$ can be obtained from the cost bounds in a
similar way.
\end{theorem}
\subsection{Discussion of particular cases}
\label{sec:DiscCase}
In comparison, for the single level QMC Petrov-Galerkin algorithm in
\cite{DKLNS13} to achieve $\calO(\varepsilon)$ error, its overall cost in
the case of $p_0<1$ is $\calO(\varepsilon^{-a^{\rm SL}})$, with
\begin{equation} \label{equ:aSL}
  a^{\rm SL}
  \,=\, \frac{p_0}{2-2p_0} + p_0 + \frac{d}{\tau}
\;.
\end{equation}

Assuming that $p_0,t,t',d>0$ are free variables and recalling that $\tau =
t + t'$, we discuss when the multi-level algorithm is more cost effective
than the single level algorithm, bearing in mind the constraints between
these variables which are implicit in the error bounds.

\noindent {\bf  (a)}
If $d/\tau \le p_t - d/t$, then
\begin{equation*}
 a^{\rm SL} - a^{\rm ML} \,=\,
\frac{p_0}{2-2p_0} + p_0 + \frac{d}{\tau} - p_t\;,
\end{equation*}
which is positive if
\begin{equation*} \label{equ:2a}
    \frac{d}{\tau} + \frac{d}{t} \,\le\,
    p_t \,<\, \frac{p_0}{2-2p_0} + p_0 + \frac{d}{\tau}
\;.
\end{equation*}

\noindent {\bf  (b1)}
If $p_t -d/t < d/\tau < p_t$ and $d/t\le p_0/(2-2p_0)$, then
\begin{equation*}
  a^{\rm SL} - a^{\rm ML} \,=\,
  p_0 + \Big(\frac{p_0}{2-2p_0} - \frac{d}{t}\Big) \,>\, 0\;.
\end{equation*}

\noindent {\bf  (b2)} 
If $p_t - d/t < d/\tau < p_t$ and $d/t > p_0/(2-2p_0)$, 
then
\begin{equation*}
  a^{\rm SL} - a^{\rm ML} \,=\,
  p_0 - \Big(1 - \frac{t p_0}{d(2-2p_0)}\Big) \Big(p_t - \frac{d}{\tau}\Big)\;,
\end{equation*}
which is positive if
\begin{equation*} 
    \frac{d}{\tau} \,<\, p_t \,<\, \frac{d}{\tau} + \frac{p_0}{1 - tp_0/(2d(1-p_0))} \;.
\end{equation*}

\noindent {\bf (c)} If $d/\tau \ge p_t$, then
\begin{equation*}
  a^{\rm SL} - a^{\rm ML} \,=\,
  p_0 + \Big( \frac{p_0}{2-2p_0} - \min\Big(\frac{d}{t},\frac{p_0}{2-2p_0}\Big)\Big) \,>\, 0
\;.
\end{equation*}

We see that the multi-level algorithm outperforms the single level one
over a large range of $p_t$ and $t$. In particular, 
for $t=t'=1$ and in the symmetric case, eg.\ when continuous, 
piecewise linear Finite Elements are used to discretize the second order, 
self-adjoint elliptic PDE \eqref{eq:PDE2}, 
the multi-level algorithm $Q^*_L$ in \eqref{eq:QL*}
always outperforms the single level one when $d\ge 2$ under
Assumption~\eqref{eq:psumpsi0}.
%
\section{Numerical Experiments}
\label{sec:NumExp}
For a parameter $\bsy\in U=[-\frac12,\frac12]^\bbN$, 
in the physical domain $ D=(0,1)^2 $,
we consider the parametric diffusion equation 
\eqref{eq:PDE1} with homogeneous Dirichlet boundary conditions.
We parametrize the uncertain diffusion coefficient $a$
with the basis from \eqref{eq:sinuseig} by
\begin{align}\label{eq:coeff}
 a(\bsy)(\bsx)
 &\,=\, a_0(\bsx) + \sum_{k_1, k_2 = 1}^\infty y_{k_1, k_2}\, \frac{1}{(k_1^2+k_2^2)^2} 
\,
 \sin( k_1 \pi x_1)\, \sin(k_2 \pi x_2) \nonumber \\
 &\,=\, a_0(\bsx) + \sum_{j=1}^\infty y_j\, \lambda_j\, \sin (k_{1,j} \,\pi x_1)\, \sin(k_{2,j} \,\pi x_2) \;,
\end{align}
where the sequence of pairs 
$ \left( (k_{1,j}, k_{2,j}) \right)_{j \in \mathbb{N} }$ 
is an ordering of the elements of 
$\mathbb{N} \times \mathbb{N}$ such that 
$k_{1,j}^2+k_{2,j}^2 \le k_{1,j+1}^2 + k_{2,j+1}^2$
for all $j \in \mathbb{N}$ 
(for cases where we have equality, the ordering is arbitrary). 
Then $\lambda_j = (k_{1,j}^2 + k_{2,j}^2)^{-2} \asymp j^{-2}$ (cf. \eqref{eq:Weyl}). 
We take $a_0(\bsx) \equiv 1$. In \eqref{eq:PDE1},
we use the forcing term $f(\bsx) = 100x_1$, and we
consider the quantity of interest in \eqref{eq:int} to be the integral of
the parametric solution $u(\bsy)$ over the physical domain $D$, i.e.,
$G(u(\bsy)) = \int_D u(\bsy)(\bsx) \,\rd\bsx$. 
The problem fits into the abstract framework with 
symmetric bilinear form $\fa(\bsy;\cdot,\cdot)$,
and with $\bcX = \bcY = H^1_0(D)$, and with
\[
  d = 2, \quad
  t=t'=1, \quad
  \tau = 2, \quad 
  \mbox{and any} \quad 
  \frac{1}{2} < p_0 \leq 1, 
\]
which implies by \eqref{eq:pt} that $p_1 = p_0/(1-p_0/2) > 2/3$.
The regularity spaces in Assumption \ref{ass:XtYt} are 
$\bcX_1 = (H^1_0\cap H^2)(D)$ and $\bcY'_1 = L^2(D)$.

We compare the single level algorithm \eqref{eq:qmcG} with the
multi-level algorithm \eqref{eq:QL*}. In both algorithms, we solve
\eqref{eq:PDE1} by the finite element method with continuous, piecewise
linear elements on a family of uniform triangulations with meshwidth
$h_\ell=2^{-(\ell+1)}$ for $\ell = 0,1,2\ldots$, and we use interlaced
polynomial lattice rules with $N=2^m$ points, $m\in\bbN$, constructed by
the fast CBC algorithm for SPOD weights from \cite{DKLNS13}.
We used the pruning strategy in \cite{GS14} to ensure that no repeated generating 
components are selected.

In the single level algorithm \eqref{eq:qmcG}, the meshwidth is
$h=h_L =2^{-(L+1)}$, leading to a finite element error of $\calO(h^2)$. 
We balance this $\calO(h^2)$ discretization error with the 
dimension truncation error of $\calO(s^{-2})$ and the 
QMC quadrature error of $\calO(N^{-2})$,
yielding the choice $s=h^{-1}=2^{L+1}$ and $N=h^{-1}$, i.e.,
$m=\log_2(h^{-1})=L+1$. This yields a total error of $\calO(h^2) =
\calO(\varepsilon)$ and cost of $\calO(Nh^{-2}s) = \calO(h^{-4})= \cO(
\varepsilon^{-2})$, ignoring logarithmic factors. Specifically, the SPOD
weights that enter the fast CBC construction are given by 
\cite[Equation (3.32) with (3.17)]{DKLNS13}, with base $b=2$, 
and with
\[
\mbox{
interlacing factor $\alpha = \lfloor 1/p_0 \rfloor + 1  = 2$, 
and}\;
\beta_j = \beta_{0,j} = \lambda_j = \frac{1}{(k_{1,j}^2 + k_{2,j}^2)^2} 
\;.
\]
The generating vectors were computed by
the fast CBC construction from \cite{DKLNS13} 
with Walsh constant $C=1.0$ (computations with 
$C=0.1$ and $C=0.01$ yielded different generating vectors,
but produced essentially the same results in this example). 
For base $b=2$, the choice $C=1.0$ is 
theoretically justified in \cite{Yoshiki2015}.

In the multi-level algorithm \eqref{eq:QL*}, for given maximal level $L$,
we take bisection refinement of the simplicial mesh in $D$ with $h_\ell =
2^{-(\ell+1)}$ for $\ell=0,1,\ldots,L$, and we follow \eqref{eq:def_s} to
select the truncation dimension as $s_\ell = \min(2^{4\ell},2^L)$, and
$m_\ell = \min(20,\lceil\log_2(N_\ell)\rceil)$, 
where by using \eqref{eq:def_N0} and \eqref{eq:def_N} 
for this particular case
$$
  N_0    = 2^{2L}, \;\; N_\ell = (2^{ 2(L-2\ell)} s^{-1}_\ell)^{2/5}
\;.
$$
By Theorem~\ref{thm:summary}, 
using formally the limiting values $p_0=1/2$ and $p_1 = 2/3$,
the total error is $\calO(h_L^2) = \calO(\varepsilon)$ at cost of 
$\calO\left(\sum_{\ell=0}^L N_\ell h_\ell^{-2} s_\ell \right) 
 = \calO(\varepsilon^{-3/2})$, ignoring logarithmic factors. 
The SPOD weights that enter the fast CBC construction are different from those
for the single-level algorithm; they are given by
\eqref{eq:choiceweight-1} and \eqref{eq:beta}. 
Again we take base $b=2$ and Walsh constant $C_{\alpha,b}=1$, 
but now with
\[
  \mbox{
  interlacing factor $\alpha = \lfloor 1/p_1 \rfloor + 1  =  2$, and}\;
  \beta_j = \beta_{1,j} = \lambda_j\,\pi\,\max(k_{1,j}, k_{2,j})
\;.
\]
In the QMC rules used in these experiments, 
we have taken in the definition \eqref{eq:beta} for the weights
$\beta_j$ to be $\beta_{1,j}$ rather 
than the precise maximum in \eqref{eq:beta}.

We remark that the error bound \eqref{eq:error-combined} allows
us to attain aforementioned convergence rates even by using on 
level $\ell = 0$ QMC quadratures with the SPOD weight sequence 
$\gamma_{\setu} =  \sum_{\bsnu_\setu\in\{1:\alpha\}^{|\setu|}}\!\!
  |\bsnu_\setu|!\,\prod_{j\in\setu} \big(2^{\delta(\nu_j,\alpha)}\beta_{0,j}^{\nu_j}\big)$ 
(cp. \eqref{eq:choiceweight-1}).
Using the (conservative) choice $\gamma_{\setu}$ from \eqref{eq:beta}
on \emph{all} discretization levels resulted 
in essentially the same numerical results.

We compute the solution up to level $L=8$, yielding $s=256$ active
dimensions. The reference solution was computed on level $L=9$ with
truncation dimension $s=1024$ and $N=2^{20}$ QMC points.
In Figure \ref{fig:convergence}, we used the work measures
$W_{\rm SLQMC} := h_L^{-2} s N$ 
and 
$W_{\rm MLQMC} := \sum_{\ell=0}^L N_\ell h_\ell^{-2} s_\ell$.

\begin{figure}[h!]
    \centering
    \includegraphics[width=\textwidth]{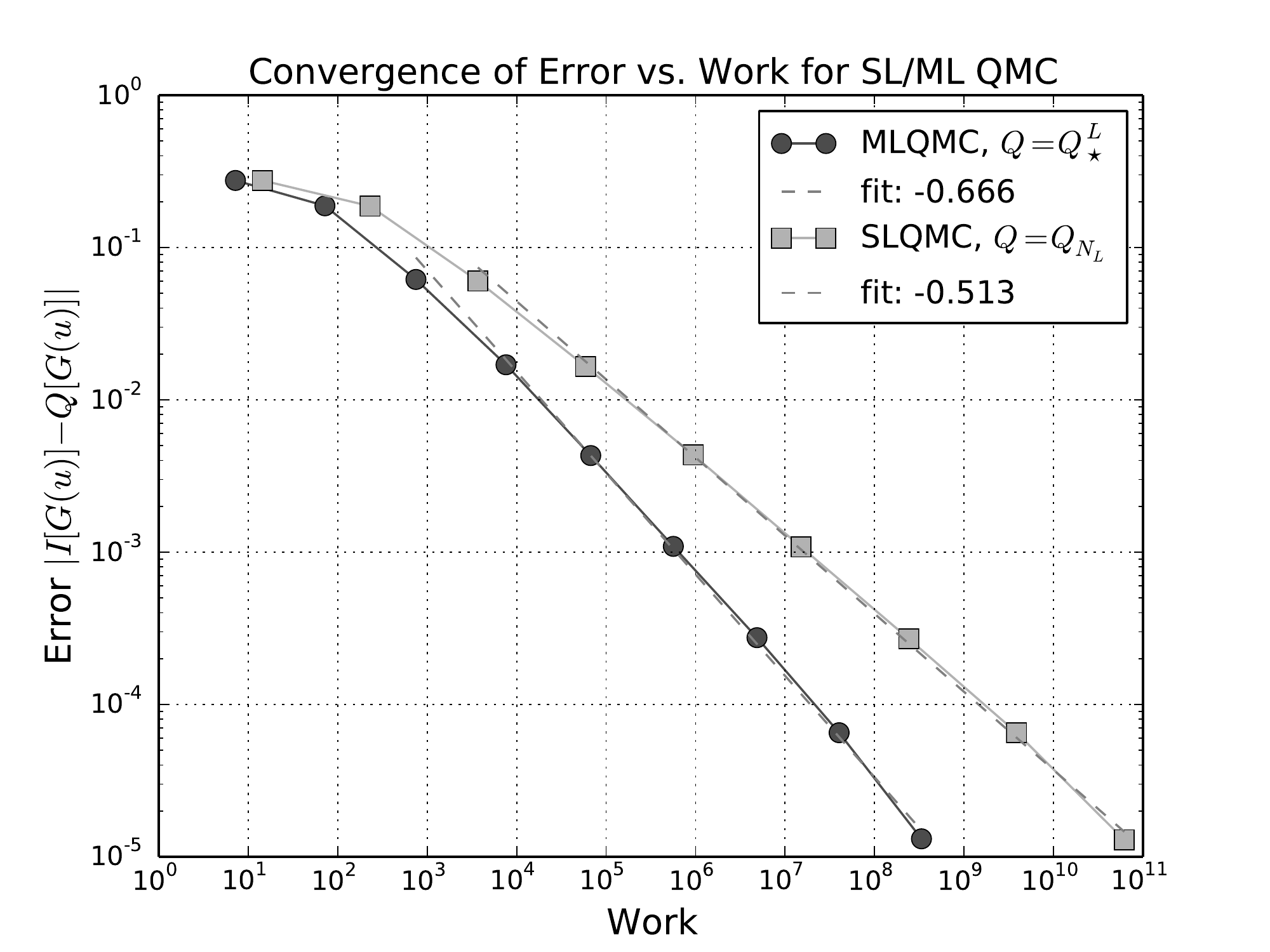}
    \caption{Convergence of the error vs. the work. 
    The theoretical rates are $-2/3$ for MLQMC and $-1/2$ for SLQMC.
    The slopes were computed by a linear fit using the last five measurements.}
    \label{fig:convergence}
\end{figure}
%
\section{Conclusions}
\label{sec:ConclGen}
We designed and analyzed a multi-level QMC Petrov-Galerkin discretization
for the approximate evaluation of functionals of solutions of countably
affine parametric operator equations. The presently proposed
algorithms extend on the one hand the single level higher order QMC
algorithms proposed in \cite{DKLNS13}, and on the other hand generalize
the multi-level approach of \cite{KSS13} from first order finite
elements and first order randomly-shifted lattice rules to higher order
in both cases.
At the same time, the class of admissible operator
equations covered by our analysis is considerably larger, allowing in
particular also indefinite, elliptic systems in non-smooth domains and
space-time Galerkin discretizations of linear parabolic evolution
problems.
Numerical tests confirmed the theoretical results, and indicate
that the presently obtained combined error bounds are
attained in the practical range of discretization parameters,
and that they can be used for practical algorithm design.
\section*{Acknowledgements}
%
Frances Kuo is the recipient of an Australian Research Council Future Fellowship (FT130100655). The research of the
first, second, and third authors was supported under the Australian
Research Council Discovery Projects funding scheme (project DP150101770).
This work was initiated while Christoph Schwab visited University of New
South Wales during the fall of 2013, while being supported in part by the
European Research Council (ERC) under AdG247277. The numerical results
presented in Section~\ref{sec:NumExp} have been performed by Robert N.
Gantner in his PhD research at the Seminar for Applied Mathematics
of ETH Z\"urich.

\end{document}